\documentclass{amsart}

\usepackage{amscd, amssymb, amsmath, amsthm}
\usepackage{enumerate}
\usepackage{latexsym,mathrsfs}
\usepackage{graphics,graphicx,psfrag}
\usepackage[colorlinks=true]{hyperref}

\usepackage[all]{xy}

\newtheorem{theorem}{Theorem}[section]
\newtheorem{lemma}[theorem]{Lemma}
\newtheorem{proposition}[theorem]{Proposition}
\newtheorem{corollary}[theorem]{Corollary}

\theoremstyle{definition}
\newtheorem{definition}[theorem]{Definition}

\newtheorem{question}[theorem]{Question}
\newtheorem{remark}[theorem]{Remark}

\newcommand{\RR}{{\mathbf R}}

\newcommand{\QQ}{{\mathbf Q}}

\newcommand{\ZZ}{{\mathbf Z}}

\newcommand{\Hyp}{{\mathbb{H}^3}}
\newcommand{\Euc}{{\mathbb{E}^3}}
\newcommand{\Sph}{{\mathbb{S}^3}}
\newcommand{\HypEuc}{{\mathbb{H}^2\times\mathbb{E}^1}}
\newcommand{\SphEuc}{{\mathbb{S}^2\times\mathbb{E}^1}}
\newcommand{\Sft}{{\widetilde{\mathrm{SL}}_2}}
\newcommand{\Sol}{{\mathrm{Sol}}}
\newcommand{\Nil}{{\mathrm{Nil}}}

\newcommand{\vrtx}{{\mathrm{Ver}}}
\newcommand{\edge}{{\mathrm{Edg}}}

\newcommand{\SL}{{\mathrm{SL}}}

\newcommand{\Hom}{{\mathrm{Hom}}}
\newcommand{\Tor}{{\mathrm{Tors}}}

\newcommand{\Mod}{{\mathrm{Mod}}}

\newcommand{\Area}{{\mathrm{Area}}}

\newcommand{\qf}{{\mathfrak{q}}}

\newcommand{\distortion}{{\mathscr{D}}}

\newcommand{\geo}{{\mathtt{geo}}}

\newcommand{\tori}{{\mathcal{T}}}
\newcommand{\pieces}{{\mathcal{J}}}

\title[Finiteness of nonzero degree maps]{Finiteness of nonzero degree maps\\ between three-manifolds}
     
\author[Yi Liu]{Yi Liu} 
\address{%
	Beijing International Center for Mathematical Research\\
	Peking University\\
	Beijing 100871\\
	China} 
\email{%
  liuyi@bicmr.pku.edu.cn}  

\thanks{Partially supported by the Recruitment Program of Global Youth Experts of China}

\date{\today}

\begin{document}

\begin{abstract}
In this paper, it is shown 
that every orientable closed 3-manifold 
maps with nonzero degree onto at most finitely many homeomorphically
distinct irreducible non-geometric orientable closed 3-manifolds.
Moreover, given any nonzero integer,
as a mapping degree up to sign,
every orientable closed 3-manifold maps with that degree onto
only finitely many homeomorphically distinct 
orientable closed 3-manifolds.
\end{abstract}

\maketitle


\section{Introduction}\label{Sec-Intro}

Let $M$ and $N$ be two orientable closed $3$--manifolds. 
For any nonzero integer $d$,
we say that $M$ \emph{$d$--dominates} $N$ if
there exists a map $f\colon M\to N$ 
of degree $\pm d$ up to sign. 
We say that $M$ \emph{dominates} $N$
if $M$ $d$--dominates $N$ for some nonzero integer $d$.

In this paper, we prove the following result:

\begin{theorem}\label{main-dominate} 
Every orientable closed $3$--manifold dominates at most finitely many homeomorphically
distinct irreducible non-geometric $3$--manifolds.\end{theorem}

In \cite{BRW}, M.~Boileau, H.~Rubinstein and S.~Wang asked the question:

\begin{question}\label{questionBRW} Does every orientable closed $3$--manifold dominate
at most finitely many irreducible homeomorphically distinct 
$3$--manifolds which support none of the geometries $\Sph$, $\Sft$,
or $\Nil$?
\end{question}

Note that any orientable closed $3$--manifold which supports the geometry
$\Sph$, $\Sft$, or $\Nil$ 
dominates infinitely many homeomorphically distinct $3$--manifolds,
which support the same geometry. 
Combined with known results for geometric targets, 
Theorem \ref{main-dominate} completes a positive answer to Question \ref{questionBRW},
(see Section \ref{Sec-geomPieces} for details). 
Moreover, we derive the following corollary, which answers 
an earlier question of Y.~Rong \cite[Problem 3.100]{Ki}:

\begin{corollary}\label{main-d-dominate} 
For any nonzero integer $d$,
every orientable closed $3$--manifold $d$--dominates only finitely many homeomorphically
distinct $3$--manifolds.\end{corollary}

Rong's original question was asked only about $1$--domination. 
By around 2002, many partial results had been proved, 
including a complete affirmative answer to Rong's
question for the target-geometric case, (see \cite{So,WZ}, etc.).
Those results clarified what kind of finiteness should be expected,
and motivated Question \ref{questionBRW}. 
We refer the reader to Wang's survey \cite{Wa};
see also the introduction of \cite{BRW} for more recent results.

The technical core for proving Theorem \ref{main-dominate} 
is a presentation length estimation, 
as developed in \cite{AL}, (see Proposition \ref{looplessGraphCase}). 
However, the proof of \cite{AL} 
relies on a special trick
for knot complements called desatellite.
In more general settings, estimates for complexity of gluings 
does not simply reduce to complexity of geometric pieces.
Indeed, finiteness of gluings becomes an essential issue 
for the present paper, (see Section \ref{Sec-geomPieces}).
To illustrate the idea, let us assume for the moment
that $M$ is an orientable closed $3$--manifold 
that dominates an orientable closed $3$--manifold
$N$,
and require that $N$ is obtained 
by gluing two one-cusped hyperbolic $3$--manifolds 
along their toral boundaries. 
After triangulating $M$ and
pulling straight $f$ with respect to the hyperbolic geometry
on both of the pieces, 
the area of the $2$--skeleton image $f(M^{(2)})$ should 
be at least the area in the neighborhood of the cutting torus $T\subset N$,
approximately the sum of its areas within the (Margulis) horocusps. 
When the gluing is complicated enough,
there could be at most one slope $\alpha$ on $T$ 
which is short on both sides,
(that is, no longer than some given bound
imposed by $M$). 
Let us pretend as if we could move $f(M^{(2)})\cap T$ into
 a regular neighborhood of $\alpha$ by homotopy on $T$.
Then $f\colon M\to N$ would factor through the drilling
$N-\alpha$ homotopically, 
violating the nonzero degree assumption. 
Therefore, the above argument leads to
an \textit{a priori} upper bound 
for the complexity of gluings assuming domination.
Similar as in \cite[Theorem 3.2]{AL}, 
the hypothetical factorization
does not truly exist in general.
However, it works in a certain homological sense,
which turns out to be sufficient for our application. 
When Seifert fibered pieces are involved, 
bounding only the local complexity of gluings 
near the cutting tori is not enough for inferring
finiteness of the allowable targets.
We handle that case by bounding the total complexity 
of the gluings at all the boundary components 
for each Seifert fibered piece. 
The spirit remains the same as the local case.
The actual proof of Theorem \ref{main-dominate}
is generated from the above idea, 
and with some significant simplification
that bypasses the factorization argument.
In fact, we make instead 
a Poincar\'{e}--Lefschetz duality argument,
available thanks to the nonzero degree assumption,
(see Lemma \ref{dominationOnto}).

In Section \ref{Sec-geomPieces}, we recall known facts
and reduce Question \ref{questionBRW}
to finiteness of gluings. In Section \ref{Sec-distortion}, 
we introduce a notion called \emph{distortion},
which provides a measurement for the complexity of gluings.
In Section \ref{Sec-JSJgluings}, 
we show that any given upper bound for the so-called 
primary average distortion 
implies gluing finiteness.
In Section \ref{Sec-nonGeomCase},
we prove Theorem \ref{main-dominate} 
by bounding the primary average distortion of any allowable targets.
In Section \ref{Sec-boundedDegree}, 
we prove Corollary \ref{main-d-dominate} 
by handling the remaining target-geometric cases,
using some generalized form of known arguments. 
In Section \ref{Sec-Conclusion}, we propose several questions
for further study.

\subsection*{Acknowledgement}
The author is grateful to his thesis advisor Ian Agol
for many helpful conversations. 
The author also thanks Hongbin Sun, Shicheng Wang 
for valuable communications,
and an anonymous referee 
for a list of suggestions 
on improving the exposition of the present paper.

\section{Background}\label{Sec-geomPieces}
In this section, 
we review known results, which essentially reduce Question \ref{questionBRW}
to finiteness of gluings.
For standard terminology and facts of $3$--manifold topology, 
see \cite{Ja,Th-notes,MF}.

Let $N$ be any orientable closed irreducible $3$--manifold. 
The Geometrization
Theorem implies that there is a canonical geometric decomposition,
which cuts $N$  into compact pieces
along a (possibly empty) minimal finite collection of 
mutually disjoint essential tori and Klein bottles.
Every piece is geometric in the sense that it supports 
one of the eight $3$--dimensional geometries of finite volume.
Moreover, the collection of cutting essential tori and Klein bottles
is unique up to isotopy. 

When $N$ is not itself geometric, the pieces are either atoroidal or Seifert fibered.
The atoroidal pieces support the $\Hyp$--geometry, and the Seifert fibered ones support
the $\HypEuc$--geometry (and also the $\Sft$--geometry). 
When $N$ is geometric, there are three cases:
$N$ is atoroidal, supporting
the $\Hyp$-geometry; 
or $N$ is Seifert fibered, supporting one of the six geometries
$\HypEuc$, $\Sft$, $\Euc$, $\Nil$, $\SphEuc$, or $\Sph$; 
or otherwise, $N$ supports the $\Sol$--geometry. 
Moreover, the geometry for the Seifert fibered case 
can be determined
by the sign of the orbifold Euler characteristic $\chi\in\QQ$ 
of the base $2$--orbifold 
together with 
the triviality of the Euler number $e\in\QQ$ of the fibration.

The geometric decomposition of $N$ is closely related to
the Jaco--Shalen--Johanson (JSJ) decomposition.
When $N$ is not $\Sol$--geometric,
all the geometric pieces of $N$ are JSJ pieces,
and the other JSJ pieces of $N$ are
compact regular neighborhoods of the cutting Klein bottles.
When $N$ is $\Sol$--geometric,
the JSJ decomposition cuts $N$ along an essential torus,
resulting in one or two JSJ pieces.
In the former case, the JSJ piece is homeomorphic to an \emph{orientable thickened torus},
by which we mean the trivial interval-bundle over a torus;
in the latter case, the JSJ pieces are both
homeomorphic to an \emph{orientable thickened Klein bottle}, 
by which we mean the interval-bundle over a Klein bottle whose bundle space is orientable.

Towards Question \ref{questionBRW}, we have the following known partial answers,
each of which captures certain finiteness about the target manifolds
assuming domination from a given $3$--manifold:

\begin{theorem}[{Boileau--Boyer--Wang \cite[Corollary 3.6]{BBW}}]\label{ThmBBW} 
Every orientable closed $3$--manifold dominates at most
finitely many homeomorphically distinct $3$--manifolds
which are $\Sol$--geometric.
\end{theorem} 

\begin{theorem}[{Boileau--Rubinstein--Wang \cite[Theorem 1.1]{BRW}}]\label{ThmBRW} 
Given any orientable closed $3$--manifold $M$,
there exists a finite collection of orientable compact $3$--manifolds 
with the following property:
For any orientable closed irreducible $3$--manifold $N$ 
which supports none of the geometries $\Sph$, $\Sft$ or $\Nil$,
if $N$ is dominated by $M$, then every JSJ piece of $N$ is
homeomorphic to one of the $3$--manifolds from the asserted collection.
\end{theorem}

\begin{lemma}[{See \cite[Lemma 4.2]{BRW}}]\label{LemmaBRW}
Let $M$ be an orientable closed $3$--manifold.
Denote by $h(M)$ the Kneser--Haken number of $M$, namely,
the maximally possible number 
of mutually disjoint, mutually non-parallel
essential subsurfaces of $M$.
Suppose that $N$ is an irreducible orientable closed $3$--manifold 
which is non-geometric.
If $N$ is dominated by $M$, 
then the geometric decomposition of $N$
yields at most $h(M)$ cutting tori and Klein bottles,
and at most $h(M)+1$ geometric pieces.
\end{lemma}

We explain how the above partial results 
motivate Theorem \ref{main-dominate} and 
motivate our plan of proof.
We observe that Theorems \ref{ThmBBW} and \ref{ThmBRW} 
has already answered Question \ref{questionBRW}
positively assuming the target $3$--manifold to be geometric,
so it remains to show finiteness of the nongeometric targets,
as stated by Theorem \ref{main-dominate}.
In this case,
the homeomorphism type of any target $3$--manifold
is determined by three parts of data:
the homeomorphism types of the geometric pieces,
and the adjacency relation between the pieces,
and the actual identification between the boundary components
of the pieces subject to the adjacency relation.
For the moment, 
let us informally refer to the former two as the preglue data,
and the last one as the gluing.
Given any $3$--manifold as the source of domination, 
Theorem \ref{ThmBRW} says that the allowable homeomorphism types of the geometric pieces
are bounded.
Lemma \ref{LemmaBRW} says that 
any allowable target must not have too many geometric pieces, 
and therefore,
the allowable types of the preglue data are bounded as well.
It is the remaining goal for the present paper  
to bound the allowable ways of the gluing.
Strictly speaking, there are certain evident operations on gluings
which keep the resulting $3$--manifold homeomorphically unchanged,
so we need to bound the allowable gluings up to certain equivalence relation.
The precise definitions are introduced in Section \ref{Subsec-gluings}.
The above reduction of the goal is stated more formally
as Proposition \ref{fixing_preglue} and Corollary \ref{fixing_preglue_corollary}.
For each equivalence class of gluings we are able to introduce 
a geometrically defined complexity, called the primary average distortion.
This is done in Section \ref{Sec-distortion}, 
see Definition \ref{primaryDistortion}.
We show finiteness of the equivalence classes 
under the assumption of uniformly bounded complexity (Proposition \ref{finiteGluings}).
We obtain a uniform complexity bound, 
which relies only on the homeomorphism type of the given dominating $3$--manifold,
(Proposition \ref{distortionDomination}).
Then Theorem \ref{main-dominate} follow from the above steps.
A summary of the proof is included at the end of Section \ref{Sec-nonGeomCase}.

%

\section{Distortion of $3$--manifolds}\label{Sec-distortion}
In this section, we introduce a geometrically defined quantity
called primary average distortion.
For any gluing of a prescribed collection of geometric pieces
with a prescribed adjacency pattern,
the quantity measures intuitively
the extent to which the resulting $3$--manifold 
is not geometric.
The primary average distortion for an orientable closed irreducible $3$--manifold
is therefore defined by reversing the process of geometric decomposition.

\subsection{Gluings of geometric pieces}\label{Subsec-gluings}
With the geometric decomposition in mind, 
we formally define gluings, for
any prescribed collection of geometric pieces
with a prescribed adjacency pattern.
The prescribed data is put together as
what we call a preglue graph of geometrics.
Roughly speaking, it is a certain graph whose vertices 
are decorated by the geometric pieces to be glued up,
and whose edges indicate the pairing relationship 
between the boundary components to be glued along.
However, we have to allow generalized edges which are one-ended,
so as to represent cutting Klein bottles.
Meanwhile, we need to distinguish some vertices,
whose decorating pieces contain embedded Klein bottles,
because they look different from others in the subsequent treatments.

Throughout this paper, we use the term \emph{graph}
in a generalized sense, 
allowing semi vertices and semi edges.
Any such graph $\Lambda$ can be modeled
as the orbit space of a cellular $1$--complex with a combinatorial involution 
which fixes no $1$--cells.
(The cellular complex and the involution are considered to be data.)
The set of \emph{vertices} $\vrtx(\Lambda)$ consists of 
the orbits of the $0$--cells.
A vertex is said to be \emph{entire}
or \emph{semi} according to the order of the $0$--cell orbit,
being $2$ or $1$ respectively.
The set of \emph{edges} $\edge(\Lambda)$
consists of the orbits of the $1$--cells.
An edge is said to be \emph{entire}
or \emph{semi} according to the order of the $1$--cell orbit,
being $2$ or $1$ respectively. 
The set of \emph{ends of edges} $\widetilde\edge(\Lambda)$
consists of the orbits of ends of $1$--cells.
Every edge end is \emph{owned by} a unique edge,
and is \emph{adjacent to} a unique vertex, in the obvious sense.
The \emph{valence} of a vertex is the number of 
the ends of edges which are adjacent to it. 
The ownership by edges yields a canonical projection
$\widetilde\edge(\Lambda)\to \edge(\Lambda)$,
which can be considered as a double branched covering
ramifying precisely at the semi edges.
The covering transformation is therefore an involution
$\widetilde\edge(\Lambda)\to\widetilde\edge(\Lambda)$,
which fixes the ends of semi edges and which switches 
the pairs of ends of entire edges.
For any end of edge $\delta\in\widetilde\edge(\Lambda)$,
we denote by $\bar\delta$ its image under the covering transformation,
and refer to $\bar\delta$ as the \emph{opposite end} of $\delta$.

One may alternatively model a graph with semi vertices and semi edges
as a cellular $1$--complex with some dead-end $0$--cells marked as
semi midpoints, and with the rest $0$--cells marked 
as either semi or entire vertices.
The $1$--cells attached to the former $0$--cells
are therefore marked as semi edges, and the rest $1$--cells
are marked as entire edges.
See Figure \ref{figOrbigraph}.
The orbit space model can be recovered from the alternative model by
double branched covering 
ramifying at the semi vertices and the semi midpoints,
and conversely by collapsing the orbits of points. 

\begin{figure}[htb]
\centering
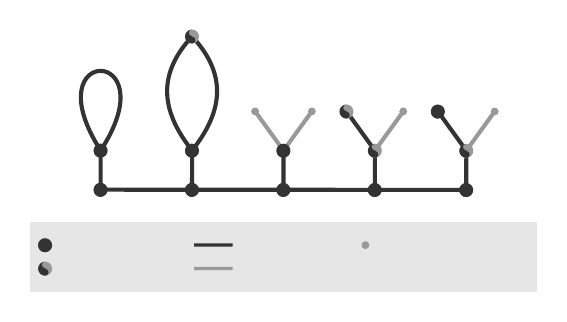
\caption{
A graph with 9 entire vertices, 4 semi vertices,
14 entire edges, and 4 semi edges
}\label{figOrbigraph}
\end{figure}


\begin{definition}\label{preglue} 
A \emph{preglue graph of geometrics} is 
defined to be a connected finite graph $\Lambda$
together with an assignment for $\vrtx(\Lambda)$ and $\widetilde\edge(\Lambda)$
as follows.
Every vertex 
$v\in\vrtx(\Lambda)$ 
is assigned with an oriented compact geometric $3$--manifold $J_v$.
We require that the components of $\partial J_v$ are all incompressible tori,
and that the number of components of $\partial J_v$ equals the valence $n_v$ of $v$.
We also require that $J_v$ contains an embedded Klein bottle 
if and only if $v$ is a semi vertex.
Every edge end $\delta\in\widetilde\edge(\Lambda)$ 
which is adjacent to $v$ 
is assigned with a distinct component $T_{\delta}$ of $\partial J_v$,
equipped with the induced orientation. 
Let $\pieces$ be the disjoint union of the pieces $J_v$
for all $v\in\vrtx(\Lambda)$.
Unless otherwise mentioned,
we always use the defining pair
$(\Lambda,\pieces)$ to denote a preglue graph of geometrics, 
with the assignment implicitly assumed.
Two preglue graphs of geometrics $(\Lambda,\pieces)$ and $(\Lambda',\pieces')$
are said to be \emph{isomorphic} if there are a combinatorial isomorphism
$\Lambda\to\Lambda'$ of graphs and a homeomorphism $\pieces\to\pieces'$
which respects the assignments.
\end{definition}

\begin{remark} 
A vertex $v\in\vrtx(\Lambda)$ is semi if and only if
the geometric piece $J_v$ is Seifert fibered 
over a non-orientable $2$--orbifold.
\end{remark}

\begin{definition}\label{gluing} 
A \emph{gluing} for a preglue graph of geometrics
$(\Lambda,\pieces)$ is defined to be
an orientation-reversing involution
$$\phi\colon \partial\pieces\to \partial\pieces$$ 
with the following properties.
For every edge end $\delta\in\widetilde\edge(\Lambda)$,
the restriction of $\phi$ to the torus component $T_\delta$
is an orientation-reversing homeomorphism onto the torus assigned to the opposite end,
denoted as $\phi_\delta\colon T_{\delta}\to T_{\bar\delta}$.
Note that $\phi_{\bar\delta}=\phi^{-1}_{\delta}$
holds for all $\delta\in \widetilde\edge(\Lambda)$. 
We treat isotopic involutions as above 
to be the same gluing.
Denote by $\Phi(\Lambda,\pieces)$
the set of all the (isotopy classes of) gluings for $(\Lambda,\pieces)$. 

A gluing $\phi$ is said to be \emph{nondegenerate} 
if it does not send ordinary fibers to ordinary fibers, up to isotopy, 
for any pair of (possibly the same) adjacent Seifert-fibered pieces.
\end{definition} 

Denote by $\Mod(\partial\pieces)$ 
the special mapping class group of $\partial\pieces$,
namely, the group of component-preserving, orientation-preserving 
self-homeomorphisms of $\partial\pieces$
modulo isotopy.
There is a natural (right) action of $\Mod(\partial\pieces)$ on the space of gluings
$\Phi(\Lambda,\pieces)$. 
To be precise,
for any $\tau\in\Mod(\partial\pieces)$ and $\phi\in\Phi(\Lambda,\pieces)$,
the transformed gluing $\phi^\tau\in\Phi(\Lambda,\pieces)$ can be defined by:
$$\phi^\tau=\tau^{-1}\circ\phi\circ\tau,$$
abusing the notations for isotopy classes and their representatives.
For any $\delta\in\widetilde\edge(\Lambda)$,
we have $(\phi^\tau)_{\delta}=\tau^{-1}_{\bar\delta}\circ\phi_\delta\circ\tau_\delta$,
where $\tau_\delta\in \Mod(T_\delta)$ stands for the restriction of $\tau$ to $T_\delta$. 
The action of $\Mod(\partial\pieces)$ on $\Phi(\Lambda,\pieces)$
is clearly well-defined and transitive.

\begin{definition}\label{eqvGluing} 
Two gluings $\phi, \phi'\in\Phi(\Lambda,\pieces)$ are said to be
\emph{equivalent} if 
$\phi'$ equals $\phi^\tau$
for some homeomorphism $\tau\in\Mod(\partial\pieces)$ 
that can be extended over $\pieces$ 
to be a self-homeomorphism. 
\end{definition}

Every gluing $\phi\in\Phi(\Lambda,\pieces)$
yields a naturally associated orientable closed $3$--manifold, 
which we denote as
$$N_\phi=\pieces/\phi.$$
Namely, $N_\phi$ is 
the quotient space of $\pieces$ by identifying points of $\partial\pieces$ with their
images under $\phi$. 
It is clear 
by definition 
that $N_\phi$ has the same geometric decomposition as prescribed by $(\Lambda,\pieces)$
if and only if $\phi$ is nondegenerate.

The homeomorphism type of the $3$--manifold $N_\phi$ is clearly determined by 
the isomorphism class of the preglue graph of geometrics $(\Lambda,\pieces)$
together with the equivalence class of the gluing $\phi$.
In particular, 
fixing any preglue graph of geometrics,
we see that equivalent gluings yield homeomorphic $3$--manifolds.

\begin{proposition}\label{fixing_preglue}
	Let $M$ be an orientable closed $3$--manifold. 
	There are at most finitely many 
	distinct isomorphism classes of preglue graphs of geometrics $(\Lambda,\pieces)$
	with the following property:
	The graph $\Lambda$ is nontrivial (not a single vertex without any edge),
	and	the $3$--manifold $M$ dominates the glued-up $3$--manifold $N_\phi$ 
	for at least one nondegenerate gluing $\phi\in\Phi(\Lambda,\pieces)$.
\end{proposition}

\begin{proof}
	If $M$ dominates some $N_\phi$ as in the property, 
	the geometric decomposition of $N_\phi$ is nontrivial,
	and the geometric pieces are listed by $\pieces$ with the adjacency relation
	indicated by $\Lambda$. By Theorem \ref{ThmBRW}, 
	there is a finite collection of candidate pieces,
	determined by $M$,
	and any component of $\pieces$ subject to the property
	is homeomorphic to one of the candidates.
	By Lemma \ref{LemmaBRW}, the vertex number and the edge number of $\Lambda$ are bounded
	in terms of $M$. 
	It follows that there is a finite collection of candidate graphs,
	and any $\Lambda$ subject to the property is combinatorially isomorphic to one of the candidates.
	Moreover, there are only finitely many ways to decorate the vertices and the ends of edges
	of the candidate graphs with the candidate pieces.
	This shows the finiteness of
	the preglue graph of geometrics 
	$(\Lambda,\pieces)$ up to isomorphism subject to the property.
\end{proof}

We record the following working corollary, which reduces 
the proof of Theorem \ref{main-dominate} to the finiteness of 
allowable gluings up to equivalence.

\begin{corollary}\label{fixing_preglue_corollary}
	The statement of Theorem \ref{main-dominate} holds true
	if the following statement holds true:
	For any orientable closed $3$--manifold $M$ and any preglue graph of geometrics $(\Lambda,\pieces)$,
	there are at most finitely many distinct equivalence classes
	of nondegenerate gluings $\phi\in\Phi(\Lambda,\pieces)$
	such that $M$ dominates the glued-up $3$--manifold $N_\phi$.
\end{corollary}

\begin{proof}
	Given any $3$--manifold $M$, we obtain a finite collection of nongeometric glued-up $3$--manifolds $N_1,\cdots,N_s$
	by Proposition \ref{fixing_preglue} and the hypothetical statement of Corollary \ref{fixing_preglue_corollary}.
	We may discard redundant ones and assume that $N_1,\cdots,N_s$ are homeomorphically distinct.
	Observe that any orientable closed irreducible nongeometric $3$--manifold $N$ 
	is homeomorphic to some $3$--manifold $N_\phi$, 
	which is obtained from a nondegenerate gluing of a preglue graph of geometrics $\phi\in(\Lambda,\pieces)$.
	Then $N$ has to be homeomorphic to some $N_i$ of the above if $M$ dominates $N$.
	In other words, 
	$N_1,\cdots,N_s$ is the homeomorphically distinct finite collection 
	as asserted by Theorem \ref{main-dominate}.
\end{proof}

\subsection{Quadratic forms associated to gluings}\label{Subsec-qfGluings}
For any gluing $\phi\in\Phi(\Lambda,\pieces)$,
we introduce a distinguished quadratic form 
$\qf_\phi$ on the real vector space $H_1(\partial\pieces;\RR)$,
which is $\phi$--invariant and positive semidefinite.
It is positive definite if and only if $\phi$ is nondegenerate.
The quadratic form $\qf_\phi$ is constructed 
as follows.

Given any preglue graph of geometrics $(\Lambda,\pieces)$,
first we introduce a distinguished 
positive semidefinite quadratic form $\qf_J$
on $H_1(\partial J;\RR)$ for each component $J$ of $\pieces$.
We assume $\partial\pieces\neq\emptyset$, and hence $\partial J\neq\emptyset$,
otherwise there is nothing to construct.
The quadratic form $\qf_J$ is constructed as the sum
of its restrictions to the direct-summands $H_1(T;\RR)$ for all
the components $T$ of $\partial J$.
There are two cases according to the geometric topology of $J$.

If $J$ is atoroidal, 
the interior of $J$ admits a complete Riemannian metric 
which is hyperbolic of finite volume.
The metric is unique up to isotopy by Mostow's Rigidity Theorem.
For any component $T$ of $\partial J$, 
the restriction of $\qf_J$ to $H_1(T;\RR)$ is determined by
a distinguished translation-invariant Euclidean metric on the affine torus $H_1(T;\RR)/H_1(T;\ZZ)$,
which has systole $1$ and which represents the marked conformal class of the corresponding cusp.
To be more explicit,
denote by $J_\geo$ the interior of $J$ with the hyperbolic metric. 
For any sufficiently small constant $\epsilon>0$, the compact $\epsilon$--thick part
$J^\epsilon_\geo$ is obtained from $J_\geo$ by removing mutually disjoint horocusps.
The boundary component $\partial_T J^\epsilon_\geo$ that corresponds to $T$
is Euclidean with the induced Riemannian metric.
Moreover, by rescaling so that the shortest simple closed geodesic has length $1$,
the rescaled Euclidean metric for $\partial_T J^\epsilon_\geo$
is independent of $\epsilon$.
As $\partial_T J^\epsilon_\geo$ is naturally homeomorphic to $T$ up to isotopy,
the tangent space of the rescaled Euclidean torus $\partial_T J^\epsilon_\geo$
at any point can be identified canonically with $H_1(T;\RR)$,
by requiring that the exponential map sends $H_1(T;\ZZ)$ to the same point.
Then we construct the restriction of $\qf_J$ on $H_1(T;\RR)$
as the positive definite quadratic form
associated to the rescaled Euclidean metric form on $H_1(T;\RR)$.

If $J$ is Seifert fibered, 
for any component $T$ of $\partial J$,
we construct the restriction of $\qf_J$ to $H_1(T;\RR)$
as the square of the intersection pairing with the ordinary fiber class.
To be more explicit,
observe the canonical short exact sequence of groups:
$$1\longrightarrow\pi_1(S^1)\stackrel{i}\longrightarrow\pi_1(J)\stackrel{p}{\longrightarrow}\pi_1(\mathcal{O})\longrightarrow1,$$ 
induced from the Seifert fibration, 
where $\mathcal{O}$ stands for the hyperbolic base $2$--orbifold. 
Identify $\pi_1(T)$ as a subgroup of $\pi_1(J)$ by the inclusion,
which is unique up to conjugacy.
For any element $\zeta\in\pi(T)$, 
$p(\zeta)$ is either trivial or a positive power of a primitive element of $\pi_1(\mathcal{O})$,
so the divisibility of $p(\zeta)$ is defined to be either $0$ or the power, respectively.
By idenfitying $\pi_1(T)$ as the integral lattice of $H_1(T;\RR)$,
the restriction of $\qf_J$ to $H_1(T;\RR)$ is uniquely determined
by the property that $\qf_J(\zeta)$ equals the square of the divisibility of $p(\zeta)$.

For any gluing $\phi\in\Phi(\Lambda,\pieces)$,
the distinguished quadratic form $\qf_\phi$ on $H_1(\partial\pieces;\RR)$
 can be constructed as the sum $\qf_{\mathcal{J}}+\qf_{\mathcal{J}}\phi_*$,
where $\qf_{\mathcal{J}}$ stands for the quadratic form on $H_1(\partial\pieces;\RR)$
given by the direct sum of all the components $\qf_J$,
and
where $\phi_*\colon H_1(\partial\pieces;\RR)\to H_1(\partial\pieces;\RR)$
stands for the induced involution.
The following definition unwraps the formula with more details.

\begin{definition}\label{qfGluing} 
Given a preglue graph of geometrics $(\Lambda,\pieces)$,
suppose that $\phi\in\Phi(\Lambda,\pieces)$ is a gluing. 
For any end of an edge $\delta\in\widetilde\edge(\Lambda)$, 
denote by $v,v'\in\vrtx(\Lambda)$ 
the vertices adjacent to $\delta$ and its opposite $\bar\delta$, respectively.
We assign the restriction of $\qf_\phi$ to $H_1(T_\delta;\RR)$ by
$$\qf_\phi(\zeta)=\qf_{J_v}(\zeta)+\qf_{J_{v'}}(\phi_\delta(\zeta))$$
for all $\zeta\in H_1(T_\delta;\RR)$.
Note that this is well-defined when $\delta$ equals $\bar\delta$. 
The distinguished quadratic from $\qf_\phi$ on 
$$H_1(\partial\pieces;\RR)=\bigoplus_{\delta\in\widetilde\edge(\Lambda)} H_1(T_\delta;\RR)$$
is defined as the sum of the assigned restrictions to the direct summands.
\end{definition} 

It follows immediately from the construction that $\qf_\phi$ is positive definite 
if and only if $\phi$ is nondegenerate. 
We also observe that equivalent gluings induce the same quadratic form.

\subsection{Distortion of gluings}\label{Subsec-distortionGluings}
Given a preglue graph of geometrics $(\Lambda,\pieces)$,
we introduce the primary average distortion 
for any gluing $\phi\in\Phi(\Lambda,\pieces)$.
We start by introducing the average distortions of a gluing
at vertices and along edges.
Intuitively, these quantities
measure the local complexity of the gluing 
around the mentioned places,
(or the local obstruction to extending geometry across those places).

For any finitely generated free $\ZZ$--module $V$
and any real quadratic form $\qf$ 
on the real vector space $V_\RR=V\otimes_\ZZ\RR$,
we denote by
$$\Delta(V,\qf)\in\RR$$
the discriminant, namely,
the determinant of the matrix 
of the associated symmetric bilinear form for $\qf$ 
over any basis of $V$.
When $\qf$ is positive definite, 
$\Delta(V,\qf)$ equals the square of the volume of the Euclidean torus
$V_\RR\,/\,V$ (of dimension the rank of $V$),
whose Euclidean structure is given by the induced inner product of $\qf$.
 
\begin{definition}\label{edgeDistortion}
Let $\phi\in\Phi(\Lambda,\pieces)$ be a gluing, and let $e\in\edge(\Lambda)$ be an edge. 
We define the \emph{average distortion} (or simply, the \emph{distortion}) of $\phi$ along $e$ to be
$$\distortion_e(\phi)=\Delta\left(H_1(T_\delta;\ZZ),\qf_\phi\right)^{\frac14},$$
where $\delta$ stands for any end of $e$. Note the definition does not depend on the choice of the end.
\end{definition}

\begin{definition}\label{vertexDistortion}
Let $\phi\in\Phi(\Lambda,\pieces)$ be a gluing, and let $v\in\vrtx(\Lambda)$ be a vertex of valence $n_v$.
Suppose $n_v>0$. 
If $v$ is an entire vertex, we define the \emph{average distortion} (or simply,
the \emph{distortion}) of $\phi$ at $v$ to be
$$\distortion_v(\phi)=\Delta\left(\partial_*H_2(J_v,\partial J_v;\ZZ),\qf_\phi\right)^{\frac1{2n_v}},$$
where $\partial_*H_2(J_v,\partial J_v;\ZZ)$ stands for
the image of $H_2(J_v,\partial J_v;\ZZ)$ in $H_1(\partial\pieces;\RR)$
under the boundary homomorphism. 
If $v$ is a semi vertex, $J_v$ is Seifert fibered with a nonorientable base $2$--orbifold. 
Denote by $\tilde{J}_v$ the $2$--fold covering space of $J_v$ which corresponds to centralizer of its ordinary fiber.
Denote by $\tilde\qf_\phi$ the quadratic form on $H_1(\partial \tilde{J}_v;\RR)$
which equals the sum of its restrictions to the direct summands $H_1(\tilde{T};\RR)$,
for all the components $\tilde{T}$ of $\partial\tilde{J}_v$,
and 
which equals the pull-back of $\qf_\phi$ restricted to  $H_1(\tilde{T};\RR)$.
We define
$$\distortion_v(\phi)=\Delta\left(\partial_*H_2(\tilde{J}_v,\partial\tilde{J}_v;\ZZ),\tilde\qf_\phi\right)^{\frac1{4n_v}}.$$
We define $\distortion_v(\phi)=0$ for $n_v=0$.
\end{definition}

\begin{remark} The edge case is similar to the vertex case.
In fact, one may take a compact regular neighborhood $\mathcal{U}_e$ of $T_e$ 
for the role of the above $J_v$, 
then the formulas of Definition \ref{vertexDistortion} 
agree with Definition \ref{edgeDistortion}.
\end{remark}

\begin{definition}\label{primaryDistortion}
Let $\phi\in\Phi(\Lambda,\pieces)$ be a gluing of a preglue graph of geometrics $(\Lambda,\pieces)$. 
We define the \emph{primary average distortion} (or simply, the \emph{primary distortion}) of $\phi$ to be
$$\distortion_\Lambda(\phi)=\max
\left\{\distortion_v(\phi),\distortion_e(\phi)\colon v\in\vrtx(\Lambda),
e\in\edge(\Lambda)\right\}.$$
For any orientable closed irreducible $3$--manifold $N$,
the geometric decomposition of $N$ determines 
a canonical preglue graph of geometrics $(\Lambda,\pieces)$ 
and a canonical gluing $\phi\in\Phi(\Lambda,\pieces)$.
We define the \emph{primary distortion} of $N$ to be
$$\distortion(N)=\distortion_\Lambda(\phi)$$
with the above notations.
\end{definition}

\begin{remark} 
To explain the names,
we actually think of the average distortion
at an entire vertex as the multiplicative average of 
``the distortions in the principal directions'' on $\partial_*H_2(J_v,\partial J_v;\RR)$,
with respect to the Euclidean norm defined by $\qf_\phi$ and the lattice
$\partial_*H_2(J_v,\partial J_v;\ZZ)$. 
The same idea applies to entire edges.
Average distortions for semi objects 
are understood likewise
by passing to a characteristic $2$--fold covering space.\end{remark}

Primary distortion measures how far a $3$--manifold is from being geometric:

\begin{lemma}\label{vanishingDistortion} For an orientable closed irreducible $3$--manifold $N$,
the primary distortion $\distortion(N)$ vanishes if and only if $N$ is geometric.\end{lemma}

\begin{proof} It follows immediately from the fact that geometric decompositions
must not match up ordinary fibers in adjacent Seifert fibered pieces. 
In fact, the associated gluing is nondegenerate if $N$ is non-geometric, 
so no average distortion vanishes
along any edge.\end{proof}

Primary distortion is preserved for finite covering of the graph:

\begin{lemma}\label{graphCovering} 
Let $\tilde{N}\to N$ be a finite covering 
onto an orientable closed irreducible $3$--manifold.
Suppose that for every geometric piece $\tilde{J}\subset \tilde{N}$,
the restricted covering projection of $\tilde{J}$
onto the underlying image $J\subset N$ 
is either a homeomorphism, or
a $2$--fold covering which corresponds 
to the ordinary-fiber centralizer.
(The latter occurs only if $J$ is Seifert fibered over a nonorientable base $2$--oribifold.)
Then $\distortion(\tilde{N})$ equals $\distortion(N)$.\end{lemma}

\begin{proof} This follows immediately from definition.
We observe that every preimage component of a geometric decomposition torus or Klein bottle of $N$
must be a geometric decomposition torus or Klein bottle of $\tilde{N}$,
(Klein bottles possibly covered by tori).
One way to see this is to write $N$ as $N_\phi$, obtained from a gluing
from preglue graph of geometrics $\phi\in\Phi(\Lambda,\pieces)$.
Then $\tilde{N}$ is obtained from the lifted gluing $\tilde{\phi}\in\Phi(\tilde{\Lambda},\tilde{\pieces})$.
Here $\tilde{\pieces}$ is taken as the disjoint union of the preimage components of $\pieces$, and
$\tilde{\Lambda}$ is determined by the adjacency relation in $\tilde{N}$ of 
the components of $\tilde{\pieces}$.
If $\tilde{\phi}$ was degenerate at some entire or semi edge,
$\phi$ would also be degenerate at the underlying edge, 
(by considering the Seifert fibered piece of $\tilde{N}$ that
contains the corresponding torus or Klein bottle).
This means that the induced map between graphs $\tilde\Lambda\to\Lambda$ 
is covering in orbi-space sense.
It is also clear that the orbi-graph covering
$\tilde\Lambda\to\Lambda$ has the same degree as $\tilde{N}$ over $N$.
So the average distortions for $\tilde{N}$ at the vertices and along the edges
form the same set of values as those for $N$.
\end{proof}

\section{Distortion and finiteness of gluings}\label{Sec-JSJgluings}
In this section, 
we provide a criterion for finiteness,
namely,
with a prescribed graph of geometrics for the geometric decomposition,
there are only finitely many homeomorphically distinct
orientable closed irreducible $3$--manifolds
for which the primary distortion is bounded by a given constant.
This is by definition a restatement of the following proposition:

\begin{proposition}\label{finiteGluings} Let $(\Lambda,\pieces)$ be a preglue graph of geometrics. 
For any constant $C>0$, 
there are at most finitely many distinct equivalence classes of nondegenerate gluings 
$\phi\in\Phi(\Lambda,\pieces)$ with the property $\distortion_\Lambda(\phi)<C$.
\end{proposition}

We prove Proposition \ref{finiteGluings} in the rest of this section. 
Our strategy is as follows. 
Using distortions along the edges, 
we bound the allowable gluings up to fiber shearings
(see Definition \ref{fiberShearing}).
Using the distortions at the Seifert fibered vertices, 
we bound the allowable indices of the needed fiber shearings.
So the allowable gluings are bounded up to equivalence,
which proves Proposition \ref{finiteGluings}.
In particular,
the distortions at the atoroidal vertices 
are not used in our proof.
This is because the distortion at any atoroidal vertex 
can be bounded in terms of the distortions along its adjacent edges,
as we explain in Subsection \ref{Subsec-atorVertexDistortion}.

\subsection{Fiber shearings}\label{Subsec-fiberShearing}
We introduce an operation 
on the space of gluings called fiber shearing. 
It modifies the resulting $3$--manifold 
by a surgery on an ordinary fiber 
of a Seifert fibered piece,
which preserves the base $2$--orbifold.

Recall that for an oriented torus $T$ 
and any slope $\gamma$ of $T$, 
the (right-hand) \emph{Dehn-twist}
along $\gamma$ is (the isotopy class of)
an orientation-preserving self-homeomorphism 
$D_\gamma\colon T\to T$ 
with the property 
$D_{\gamma*}(\zeta)=\zeta+\langle\zeta,[\gamma]\rangle\gamma$ 
for all $\zeta\in H_1(T;\ZZ)$.
Here a slope of a torus is considered to be
to be an oriented essential simple closed curve up to isotopy,
and the notation $\langle\cdot,\cdot\rangle$
stands for the intersection pairing. 
Note that $D_{\gamma}$ does not depend on the direction of $\gamma$.
For any integer $k$, the \emph{$k$--time Dehn twist} along $\gamma$
refers to the iterated homeomorphism $D^k_\gamma\colon T\to T$.

\begin{definition}\label{fiberShearing}
Let $(\Lambda,\pieces)$ be a preglue graph of geometrics. 
A \emph{fiber shearing} with respect to $(\Lambda,\pieces)$
is defined to be a transformation $\tau\in\Mod(\partial\pieces)$
with the following properties.
For any vertex $v\in\vrtx(\Lambda)$
and any edge end $\delta\in\widetilde\edge(\Lambda)$ adjacent $v$, 
the restricted transformation $\tau_\delta\in\Mod(T_\delta)$ 
is the identity if $J_v$ is atoroidal; 
or it is a $k_\delta$--time Dehn twist along an ordinary fiber,
for some integer $k_\delta$, if $J_v$ is Seifert fibered. 
The \emph{index} of $\tau$ at a Seifert fibered vertex $v$
is defined to be the integer
$$k_v(\tau)=\sum_{\delta\in\widetilde\edge(v)}k_\delta,$$
where $\widetilde\edge(v)$ stands for the set of the edge ends adjacent to $v$. 
For any gluing $\phi\in\Phi(\Lambda,\pieces)$, 
the transformed gluing $\phi^\tau\in\Phi(\Lambda,\pieces)$
is called the \emph{fiber shearing} of $\phi$ under $\tau$.
\end{definition}

Note that the fiber shearings form an abelian subgroup 
of $\Mod(\partial\mathcal{J})$.
The index is additive for compositions of fiber shearings.

\begin{lemma}\label{sameIndex} 
Two fiber shearings of a given gluing are equivalent if
their indices are equal at every Seifert fibered vertex.
\end{lemma}

\begin{proof} 
It suffices to show that a fiber shearing of index $0$
at all the Seifert fibered vertices preserves the equivalence class of
a gluing. In fact, for any pair of boundary tori $T,T'$ 
of a Seifert fibered piece $J$,
there is a properly embedded annulus $A$ 
bounded by a pair of ordinary fibers, one on each.
Since $J$ is oriented,
$A$ is two-sided, and there is a well defined 
Dehn twist on $J$ along this annulus.
The restriction of the Dehn twist along $A$
give rise to a $(+1)$--time Dehn twist on $T$ and a $(-1)$--time Dehn twist
on $T'$.
It follows that any fiber shearing of index $0$ at a Seifert fibered vertex
can be extended to be 
an orientation-preserving homeomorphism of the corresponding Seifert fibered piece,
(as the composition of some Dehn twists along annuli).
Therefore, 
any fiber shearing of index $0$ at all the Seifert fibered vertices
transforms any gluing into an equivalent one.
\end{proof}

\subsection{Distortion along edges}\label{Subsec-edgeDistortion}
We show that the distortions along the edges bound 
the nondegenerate allowable gluings up to fiber shearings
(Lemma \ref{upToFiberShearings}). 
In fact, we prove a more general result of finiteness,
for twisted sums of positive semidefinite quadratic forms (Proposition \ref{qfLemma}).
The rank--$2$ case of Proposition \ref{qfLemma} 
is sufficient for Lemma \ref{upToFiberShearings}.
It is probably a lot simpler 
to prove only for that particular rank.
However, we decide to present a proof for arbitrary rank,
as it makes the underlying structures more explicit.

%
%

For any finitely generated free
$\ZZ$--module $V$,
the special linear group $\SL(V)$ acts
naturally (from the right) on the space of 
quadratic forms on the real vector space $V_\RR=V\otimes_\ZZ\RR$,
and any $\tau\in\SL(V)$ transforms a
quadratic form $\qf$ into the composition $\qf\tau$. 
We say
that a quadratic form $\qf$ has \emph{rational kernel} with respect to $V$
if the kernel $U_\RR$ of (the associated symmetric bilinear form of) $\qf$ 
in $V_\RR$ intersects $V$ in a lattice $U$ of $U_\RR$.

\begin{proposition}\label{qfLemma} 
Let $V$ be a finitely generated free $\ZZ$--module,
and $\qf,\qf'$ be two real positive semidefinite
quadratic forms on $V_\RR$ 
having rational kernels with respect to $V$. 
Denote by $\Gamma$ the special linear group $\SL(V)$,
and by $\Gamma_\qf$ and $\Gamma_{\qf'}$ 
the stabilizers for $\qf$ and $\qf'$, respectively.

Then for any transformation $\sigma\in\Gamma$,
the discriminant $\Delta(V,\qf\sigma+\qf')$ 
depends only on the double coset $\Gamma_\qf\sigma\Gamma_{\qf'}$.
Moreover, given any constant $C>0$, there are 
at most finitely many distinct double cosets $\Gamma_\qf\sigma\Gamma_{\qf'}$
of $\Gamma$ for which the following inequality holds:
$$0<\Delta(V,\qf\sigma+\qf')<C.$$
\end{proposition}

\begin{proof}
Observe that 
$\Delta(V,\qf\tau\sigma\tau'+\qf')=\Delta(V,\qf\tau\sigma+\qf'(\tau')^{-1})=\Delta(V,\qf\sigma+\qf')$
holds for any $\tau\sigma\tau'\in\Gamma_\qf\sigma\Gamma_{\qf'}$,
so the double coset $\Gamma_\qf\sigma\Gamma_{\qf'}$
determines the discriminant $\Delta(V,\qf\sigma+\qf')$.

To prove the finiteness, we consider the unit ball $B_\sigma$
of the quadratic form $\qf\sigma+\qf'$, which depends only on
the right coset $\Gamma_\qf\sigma$.
The condition $\Delta(V,\qf\sigma+\qf')>0$ guarantees that
$\qf\sigma+\qf'$ is positive definite, so $B_\sigma$ must be compact.
(Note that the unit balls $B,B'$ of $\qf,\qf'$ may be non-compact;
see Figure \ref{figDistortionQF}.)
Under the assumption $0<\Delta(V,\qf\sigma+\qf')<C$,
we show for any compact subset $K$ of $V_\RR$
that $B_\sigma$ is contained by $K$
for at most finitely many right cosets $\Gamma_\qf\sigma$.
Then we show that 
there exists some compact subset $K=K(C,\qf,\qf')$ 
with the following property:
Every double coset $\Gamma_\qf\sigma\Gamma_{\qf'}$ with $0<\Delta(V,\qf\sigma+\qf')<C$
contains a right-coset representative 
$\Gamma_\qf\sigma\tau'$ for which $B_{\sigma\tau'}$ is contained by $K$.
To this end, we need to understand 
the action of $\Gamma_{\qf'}$ on $V_\RR$
and the transformation geometry about $B_\sigma$.

\begin{figure}[htb]
\centering
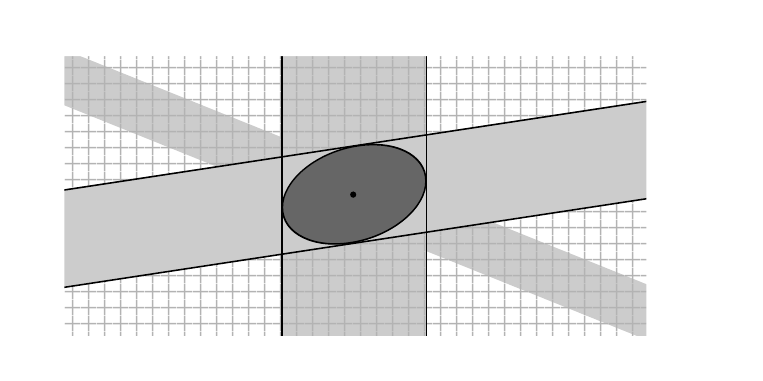
\caption{The unit balls, illustrated with $\dim V_\RR=2$ and
$\dim U_\RR=\dim U'_\RR=1$}\label{figDistortionQF}
\end{figure}

Denote by $U'_\RR$ the kernel of $\qf'$ in $V_\RR$, and by $U'$ 
the intersection $U'_\RR\cap V$. The assumption that $\qf'$ has rational kernel 
implies that $V/U'$ is embedded as a lattice of $V_\RR/U'_\RR$.
Every transformation $\tau'\in\Gamma_{\qf'}$ preserves $U'$ and descends
to an automorphism of $V/U'$ that preserves the induced positive definite quadratic form
$\overline{\qf'}$.
The stabilizer is isomorphic to 
a semidirect product of groups 
$\Gamma_{\qf'}\cong\mathrm{Hom}_{\ZZ}(V/U',U')\rtimes\left(\mathrm{GL}(U')\times\mathrm{O}\left(V/U',\overline{\qf'}\right)\right)^+$
where $\mathrm{GL}(U')\times\mathrm{O}(V/U',\overline{\qf'})$ acts on 
the abelian group $\mathrm{Hom}_{\ZZ}(V/U',U')$ by composition on both sides,
and $(\mathrm{GL}(U')\times\mathrm{O}(V/U',\overline{\qf'}))^+$ stands for 
the normal subgroup whose elements have determinant product $1$.
To be more explicit, fix a $\ZZ$--submodule $L'$ of $V$ which projects isomorphically onto
$V/U'$.
With respect to the direct-sum decomposition $V=U'\oplus L'$,
any element $\tau'\in\Gamma_{\qf'}$ is represented by a square matrix
\begin{equation*}\label{matrix_form}
\left[\begin{array}{cc}*&*\\0&*\end{array}\right]\in
\left[\begin{array}{cc} \mathrm{GL}(U') & \Hom_{\ZZ}(L',U')\\
0& \mathrm{O}\left(L',\qf'|_{L'}\right)\end{array}\right]
\end{equation*}
and of determinant $1$.
The entries here are linear homomorphisms, 
which become $\ZZ$--matrix blocks 
upon a choice of bases for $U'$ and $L'$.

The geometry of $B_\sigma$ can be understood
through its intersection with cosets of $U'_\RR$.
For any vector $\xi\in V_\RR$, the coset $\xi+U'_{\RR}$ intersects $B_\sigma$ 
in an ellipsoid $B_\sigma(\xi)$, possibly a point or empty. 
The nondegenerate ellipsoids $B_\sigma(\xi)$ are all of the same dimension as $U'_\RR$,
and all similar to  
$$B_\sigma(0)=B_\sigma\cap U'_\RR,$$
shrinking the scale. 
The affine centers of $B_\sigma(\xi)$
all lie on a unique linear subspace $H_\sigma$ of $V_\RR$, 
which is precisely the dual to $U'_\RR$ with respect to $\qf\sigma+\qf'$.
As $U'_\RR$ is the kernel of $\qf'$,
the vectors $\eta\in H_\sigma$ can be characterized 
by the property $\qf(\sigma\xi+\sigma\eta)=\qf(\sigma\xi)+\qf(\sigma\eta)$
for all $\xi\in U'_\RR$.
As $H_\sigma$ is a direct-sum complement of $U'_\RR$, 
it determines a unique homomorphism
\begin{equation*}\label{h_sigma}
h_\sigma\in\mathrm{Hom}_\RR(L'_\RR,U'_\RR)
\end{equation*}
whose graph in $V_\RR=U'_\RR\oplus L'_\RR$ equals $H_\sigma$.
Moreover, we observe the inclusion
$$B_\sigma\subset B'\cap (B_\sigma(0)+H_\sigma)$$
where $B'$ stands for the unit ball of $\qf'$ in $V_\RR$,
and $B_\sigma(0)+H_\sigma$ the convex subset of $V_\RR$
formed by adding up pairs of vectors from $B_\sigma(0)$ and $H_\sigma$.

Fix an auxiliary basis of $V$ by taking
a pair of bases for $U'$ and $L'$.
We endow $V_\RR=U'_\RR\oplus L'_\RR$ 
with an auxiliary Euclidean metric 
by requiring the fixed basis to be orthonormal.
Denote by $n$ the dimension of $V_\RR$,
and $k'$ the dimension of $U'_\RR$.
Denote by $\mu_m$ the induced Lebesgue measure
on any $m$--dimensional affine subspace of $V_\RR$.
The above inclusion implies
$$\mu_n(B_\sigma)\,\leq\,\mu_{k'}(B_\sigma(0))\cdot\mu_{n-k'}(B'\cap L'_\RR),$$
since $B'\cap (B_\sigma(0)\times H_\sigma)$ projects orthogonally onto
a subset of $B'\cap L'_\RR$ with the fibers isometric to $B_\sigma(0)$.
The volume of $B_\sigma$ can be computed by the formula
$$\mu_n(B_\sigma)\,=\,\frac{\omega_n}{\Delta(V,\qf\sigma+\qf')^{1/2}},$$
where $\omega_n=\pi^{n/2}/\Gamma(n/2+1)$ denotes the volume of an $n$--dimensional
Euclidean unit ball.
Therefore, the assumption 
$0<\Delta(V,\qf\sigma+\qf')<C$ implies
a uniform lower bound for the volume of $B_\sigma(0)$:
$$\mu_{k'}(B_\sigma(0))>\frac{\omega_n}{\mu_{n-k'}(B'\cap L'_\RR)}\cdot C^{-1/2}>0.$$

Suppose that $K$ is any compact subset of $V_\RR$.
We claim that there are at most finitely many right cosets $\Gamma_\qf\sigma$
for which $B_\sigma$ are contained by $K$,
provided $0<\Delta(V,\qf\sigma+\qf')<C$.
In fact, an ellipsoid $B_\sigma$ could be contained by $K$ 
only if the length of its shortest axis is uniformly bounded below,
since its volume must be at least $\omega_n\cdot C^{-1/2}$.
Note that $\qf\sigma+\qf'$, and hence $B_\sigma$,
is determined by the homomorphism
$\overline{\sigma}\colon V\to V/U$,
defined by $\overline{\sigma}(\xi)=\sigma(\xi)\bmod U$
where $U$ stands for the kernel of $\qf$ on $V$.
Then uniform shortest-axis estimate implies that any $\xi\in V$
admits at most finitely many possible images $\overline{\sigma}(\xi)$.
Applying to our fixed basis of $V$, we see that there are 
at most finitely many possibilities for $B_\sigma$ 
if it is contained by $K$,
so the claim follows.

It remains to argue that there exists a uniform constant $D>0$,
depending only on $C,\qf,\qf'$ and the fixed data, and
that the following statement holds true
for the Euclidean compact ball $K(D)$ of radius $D$
centered at the origin of $V_\RR$:
For any double coset $\Gamma_\qf\sigma\Gamma_{\qf'}$
with $0<\Delta(V,\qf\sigma+\qf')<C$,
there is a right-coset representative $\Gamma_\qf\sigma\tau'$,
and $B_{\sigma\tau'}$ is contained by $K(D)$.

Denote by $U_\RR$ the kernel of $\qf$ in $V_\RR$,
and $U$ the intersection $U_\RR\cap V$, and
$\overline{\qf}$ the induced positive definite quadratic form
on $V_\RR/U_\RR$.
Since $\qf\sigma+\qf'$ is positive definite,
 $\sigma(U')$ projects isomorphically onto its image 
$\overline{\sigma(U')}$ in the quotient $V/U$.
Then the discriminant $\Delta(\overline{\sigma(U')},\overline{\qf})$
is uniformly bounded above,
by our lower bound estimate for $\mu_{k'}(B_\sigma(0))$
and by the formula 
$$\mu_{k'}(B_\sigma(0))\,=\,
\frac{\omega_{k'}}{\Delta(U',\qf\sigma+\qf')^{1/2}}\,=\,
\frac{\omega_{k'}}{\Delta(\sigma(U'),\qf)^{1/2}}
\,=\,
\frac{\omega_{k'}}{\Delta(\overline{\sigma(U')},\overline{\qf})^{1/2}}.
$$
Since $\overline{\qf}$ is positive definite,
it is well-known that 
there are at most finitely many rank--$k'$ submodules of $V/U$
whose $\overline{\qf}$--discriminants are bounded by a given constant.
(This is implied by the simple facts that
$W\mapsto \wedge^{k'}W$ gives rise to a finite-to-one correspondence
between the rank--$k'$ submodules of $V/U$ and the rank--$1$ submodules of $\wedge^{k'}(V/U)$,
and that the expression
$\overline{\mathfrak{q}}_{k'}(\overline{v}_1\wedge\cdots\wedge \overline{v}_{k'})
=\Delta(\ZZ \overline{v}_1+\cdots+\ZZ \overline{v}_{k'},\overline{\qf})$
determines a unique and well-defined 
positive definite quadratic form $\overline{\qf}_{k'}$ on $\wedge^{k'}(V/U)$.)
Therefore, at most finitely many rank--$k'$ submodules
of $V/U$ may occur to be the image $\overline{\sigma(U')}$.
If two transformations $\sigma,\sigma'\in\Gamma$ give rise
to the same submodule $\overline{\sigma(U')}=\overline{\sigma'(U')}$,
the identification pulls back to be a unique isomorphism
$\tau_{U'}\in\mathrm{GL}(U')$.
Then any transformation $\tau'\in \Gamma_{\qf'}$ whose $(U',U')$--block
(of the matrix representation) equals $\tau_{U'}$
satisfies $\qf\sigma\tau'+\qf'=\qf\sigma'+\qf'$ restricted to $U'$,
or equivalently, $B_{\sigma\tau'}(0)=B_{\sigma'}(0)$.
It follows that there exist a finite collection of right cosets
$$\mathcal{S}\subset \Gamma_\qf\backslash\Gamma$$
with the following property:
For any double coset $\Gamma_{\qf}\sigma\Gamma_{\qf'}$ with $0<\Delta(V,\qf\sigma+\qf')<C$,
there are some right-coset representative $\Gamma_{\qf}\sigma\tau'$
with $\tau'\in\Gamma_{\qf'}$ and some $\Gamma_{\qf}\sigma'\in\mathcal{S}$
such that $B_{\sigma\tau'}(0)$ coincides with $B_{\sigma'}(0)$.

Furthermore, the compactness of $\mathrm{Hom}_\RR(L'_\RR,U'_\RR)/\mathrm{Hom}_\ZZ(L',U')$
enables us to strengthen the above statement: 
It can be required in addition that
$h_{\sigma\tau'}(B'\cap L'_\RR)=B'\cap H_{\sigma\tau'}$ 
is contained by some uniformly bounded neighborhood of 
$h_{\sigma'}(B'\cap L'_\RR)=B'\cap H_{\sigma'}$.
In fact, with respect to our fixed basis, we can modify 
the $(U',L')$--block of $\tau'$
so that $h_{\sigma\tau'}$ differs from 
$h_{\sigma}$ only by a $k'\times(n-k')$--matrix 
whose entries are all bounded by $1/2$ in the absolute value.
Note that modifying the $(U',L')$--block of the matrix of $\tau'$
changes $h_{\sigma\tau'}$ 
by some difference in $\mathrm{Hom}_\ZZ(L'_\ZZ,U'_\ZZ)$,
without affecting $B_{\sigma\tau'}(0)$.
Then the asserted neighborhood of $h_{\sigma'}(B'\cap L'_\RR)=B'\cap H_{\sigma'}$
can be taken, for example, to have radius 
$R=\sqrt{k'(n-k')}\cdot\mathrm{Diam}(B'\cap L'_{\RR})/4$.

Because of the finiteness of $\mathcal{S}$ and
the geometry of the ellipsoids $B_\sigma$,
we see that a uniform constant $D>0$ as asserted
can be taken,
for example,
to be the maximum of $R+\mathrm{Diam}(B'\cap H_{\sigma'})/2+\mathrm{Diam}(B_{\sigma'}(0))/2$
for all $\Gamma_\qf\sigma'\in\mathcal{S}$.
So the proof of Proposition \ref{qfLemma} is complete.
\end{proof}

\begin{lemma}\label{upToFiberShearings} 
Let $(\Lambda,\pieces)$ be a preglue graph of geometrics. 
Given any constant $C>0$, 
there are at most finitely many nondegenerate distinct gluings 
$\phi\in\Phi(\Lambda,\pieces)$ up to fiber shearings
with the following property:
For all edges $e\in\edge(\Lambda)$,
the distortion along $e$ satisfies $\distortion_e(\phi)<C$.
\end{lemma}

It might be helpful to see what we gain from Lemma \ref{upToFiberShearings},
before going into the proof.
If we denote by $\mathrm{FS}(\pieces)$ the abelian subgroup of $\Mod(\partial\pieces)$ 
formed by all the fiber shearings, 
the distortion along any edge remains constant 
on any $\mathrm{FS}(\pieces)$--orbit in the space of gluings $\Phi(\Lambda,\pieces)$.
Then Lemma \ref{upToFiberShearings} says that uniformly bounded distortion 
along all edges implies finiteness of $\mathrm{FS}(\pieces)$--orbits
whose gluings satisfy such bound.
One will obtain finiteness of equivalence classes of gluings,
as desired for Proposition \ref{finiteGluings},
if one can also bound the equivalence classes of allowable gluings for each $\mathrm{FS}(\pieces)$--orbit,
using the further assumption of uniformly bounded distortion at all vertices.
After choosing a reference gluing of the $\mathrm{FS}(\pieces)$--orbit,
this next step may be reformulated 
in terms of uniformly bounded fiber-shearing indices
at all vertices, 
thanks to Lemma \ref{sameIndex}.
We actually prove the next step in Lemma \ref{boundingFiberShearings}.

\begin{proof} 
Let $\phi\in\Phi(\Lambda,\pieces)$ be any nondegenerate gluing 
with the required property.
For any edge end $\delta\in\widetilde{\edge}(\Lambda)$, 
the restricted homeomorphism $\phi_\delta\colon T_\delta\to T_{\bar\delta}$ 
induces the quadratic form 
$\qf_\phi|=\qf_{J'}\phi_\delta+\qf_{J}$ on $H_1(T_{\delta};\RR)$,
where $J,J'$ are the pieces containing $T_\delta,T_{\bar\delta}$, respectively. 
Fix a reference gluing $\psi_\delta:T_\delta\to T_{\bar\delta}$.
Then $\phi_\delta=\psi_\delta\sigma$ for some $\sigma\in\Mod(T_\delta)$. 
Rewriting $\qf=\qf_{J}$, and $\qf'=\qf_{J'}\psi_\delta$, and $\Gamma=\Mod(T_\delta)$,
we see that $\qf_\phi$ on $H_1(T_{\delta};\RR)$ equals $\qf\sigma+\qf'$ for some $\sigma\in\Gamma$. 
The stabilizer $\Gamma_\qf$ of $\qf$ in $\Gamma$ is nontrivial 
only if $J$ is Seifert fibered, and
in this case, the stabilizer $\Gamma_\qf$
is generated by a Dehn twist along an ordinary fiber on $T_\delta$.
The stabilizer $\Gamma_{\qf'}$ is nontrivial only if $J'$ is Seifert fibered, 
and in this case, $\Gamma_{\qf'}$ is generated by a Dehn twist along an
ordinary fiber on $T_{\bar\delta}$ 
pulled back to $T_\delta$ via $\psi_\delta$. By the assumption and
the definition of edge distortion, 
we have $\Delta(H_1(T_\delta;\ZZ),\qf\sigma+\qf')<C$.
Moreover, $\Delta(H_1(T_\delta;\ZZ),\qf\sigma+\qf')>0$ 
holds because $\phi$ is nondegenerate.
Therefore,
Proposition \ref{qfLemma} implies 
that there are at most finitely many allowable types of 
$\phi_\delta$ up to fiber shearings. 
Since $\phi\colon \partial\pieces\to
\partial\pieces$ is defined by all the restricted homeomorphisms $\phi_\delta$,
for all $\delta\in\widetilde{\edge}(\Lambda)$, 
the asserted finiteness follows.
\end{proof}

\subsection{Distortion at Seifert fibered vertices}\label{Subsec-sfVertexDistortion}
We show that the distortions at the Seifert fibered vertices bound
the allowable fiber shearings of 
any given nondegenerate gluing up to equivalence.
Then we prove Proposition \ref{finiteGluings}.

\begin{lemma}\label{boundingFiberShearings} 
Let $(\Lambda,\pieces)$ be a preglue graph of geometrics,
and $\phi\in\Phi(\Lambda,\pieces)$ be a nondegenerate gluing. 
Suppose that $v\in\vrtx(\Lambda)$ is a Seifert fibered vertex.
Then given any constant $C>0$, there exists some constant
$K=K(C,\phi)>0$ with the following property:
For any fiber shearing $\phi^\tau$ of $\phi$ 
whose distortion at $v$ satisfies $\distortion_v(\phi^\tau)<C$, 
the fiber-shearing index $k_v(\tau)$ satisfies $|k_v(\tau)|<K$.
\end{lemma}

\begin{proof} There are two cases, as $v$ may be entire or semi.

\subsubsection*{Case 1}
The vertex $v$ is entire, so the Seifert fibered piece $J_v$ has an orientable base $2$--orbifold. 

In this case, we choose consistent directions for 
the fibers of $J_v$.
For any edge end $\delta$ adjacent to $v$, 
denote by $\lambda_\delta$ 
the directed ordinary-fiber slope on $T_\delta\subset\partial J_v$.
It suffices to assume the valence $n_v>0$.
There is a canonical submodule of rank $(n_v-1)$
contained by $\partial_*H_2(J_v,\partial J_v;\ZZ)$:
$$L_v=\left\{\sum_{\delta\in\widetilde\edge(v)}l_\delta\,[\lambda_\delta]
\in H_1(\partial J_v;\RR)
\colon\left.
\sum_{\delta\in\widetilde\edge(v)}l_\delta=0\textrm{ and }l_\delta\in\ZZ\right.\right\}.$$
Moreover, we can choose an element $[\mu_v]\in\partial_*H_2(J_v,\partial J_v;\ZZ)$ with the property
$$\partial_*H_2(J_v,\partial J_v;\ZZ)=L_v\oplus\ZZ\cdot[\mu_v].$$
For example, take $m_v>0$ to be the least common multiple of the cone-point orders 
of the base $2$--orbifold.
We can choose some directed slope $\mu_\delta\subset T_\delta$
for each $\delta\in\widetilde\edge(v)$, requiring the intersection number
$\langle[\mu_\delta],[\lambda_\delta]\rangle=m_v$ on $H_1(T_\delta;\ZZ)$.
Then one may check that
$$[\mu_v]=\sum_{\delta\in\widetilde\edge(v)}[\mu_\delta]$$
satisfies the above property.

Rewrite $\qf, \qf^\tau$ for $\qf_\phi,\qf_{\phi^\tau}$.
Observe $\qf^\tau=\qf$ restricted to $L_v$. 
We estimate the values of $\qf^\tau$ 
on the coset $[\mu_v]+L_v\otimes\RR$ of $\partial_*H_2(J_v,J_{\partial v};\RR)$.
For any vector $[\xi]=\sum_{\delta\in\widetilde\edge(v)}l_\delta\,[\lambda_\delta]\in L_v\otimes\RR$,
we have
$$\qf^\tau\left([\mu_v]+[\xi]\right)=
\sum_{\delta\in\widetilde\edge(v)}\qf\left([\mu_\delta]+(l_\delta+m_vk_\delta)[\lambda_\delta]\right),$$
where the integer $m_v>0$ is as above, 
and where the integers $k_\delta$ stand for the Dehn-twist numbers
on $T_\delta$ to define $\tau_\delta$ (Definition \ref{fiberShearing}).
If $|k_v(\tau)|\geq K$ holds for some given constant $K>0$,
then there must be some $\delta^*\in\widetilde\edge(v)$ which satisfies
$$|l_{\delta^*}+m_v k_{\delta^*}|\geq K/n_v,$$
thanks to the equation
$$\sum_{\delta\in\widetilde\edge(v)}\,(l_\delta+mk_\delta)=m_vk_v(\tau).$$ 
Then we can estimate:
\begin{eqnarray*}
\qf^\tau\left([\mu_v]+[\xi]\right)&\geq&\qf\left([\mu_{\delta^*}]+(l_{\delta^*}+m_vk_{\delta^*})\,[\lambda_{\delta^*}]\right)\\
&\geq&\frac12\,\qf\left((l_{\delta^*}+m_vk_{\delta^*})[\lambda_{\delta^*}]\right)-\qf\left([\mu_{\delta^*}]\right)\\
&\geq&\frac{K^2r_v}{2n_v^2}-R_v.
\end{eqnarray*}
Here the constants $r_v=\min_{\delta\in\widetilde\edge(v)}\qf([\lambda_{\delta}])$ 
and $R_v=\max_{\delta\in\widetilde\edge(v)}\qf([\mu_{\delta}])$
depend only on $J_v$ and $\phi$. 
Observe $r_v>0$ because $\phi$ is nondegenerate.
Then we can estimate:
$$\distortion_v(\phi^\tau)=\left(\Delta(L_v,\qf)\cdot \inf_{[\xi]\in L_v\otimes\RR}\{\qf([\mu_v]+[\xi])\}\right)^{\frac1{2n_v}}\geq \left(\Delta_{L_v}\cdot\left(\frac{K^2r_v}{2n_v^2}-R_v\right)\right)^{\frac{1}{2n_v}},$$
where $\Delta(L_v,\qf)$ is rewritten as $\Delta_{L_v}$.
Observe $\Delta_{L_v}>0$
because $\phi$ is nondegenerate. 
The above estimates show
that $\distortion_v(\phi^\tau)<C$ implies an upper bound for $K>0$.
In other words, under the assumption $\distortion_v(\phi^\tau)<C$,
the absolute value of the fiber-shearing index $k_v(\tau)$ 
is bounded by some $K=K(C,\phi)>0$.

\subsubsection*{Case 2}
The vertex $v$ is a semi, so the Seifert fibered piece $J_v$ 
has a non-orientable base $2$--orbifold. 

In this case, let $\tilde{J}_v\to J_v$ be the $2$--fold covering
which corresponds to the centralizer of ordinary fiber,
as appeared 
in the definition of vertex distortion (Definition \ref{vertexDistortion}). 
Then $\partial\tilde{J}_v$ 
is a trivial $2$--fold covering space of $\partial J_v$, 
and every fiber shearing $\tau\in\Mod(\partial J_v)$ at $v$ 
of index $k_v(\tau)$ 
lifts to a unique $\tilde\tau\in\Mod(\partial\tilde{J}_v)$ of index $2k_v(\tau)$. 
Then we can reduce to the previous case,
and bound $|2k_v(\tau)|$ by some $K=K(C,\phi)>0$,
since $\tilde{J}_v$ is Seifert fibered over an orientable $2$--orbifold.
\end{proof}

We are ready to summarize the proof of Proposition \ref{finiteGluings}.

\begin{proof}[Proof of Proposition \ref{finiteGluings}] By Lemma \ref{upToFiberShearings}, there are at most finitely many
 allowable types of gluings up to fiber shearings. 
By Lemma \ref{boundingFiberShearings}, for each allowable family of fiber shearings
 $\{\phi^\tau\}$,
where $\tau\in\Mod(\partial\pieces)$ runs over all the fiber shearings 
and where $\phi$ is a reference nondegenerate gluing, 
there are at most finitely many allowable indices of $\tau$
 at any Seifert fibered vertex. 
By Lemma \ref{sameIndex}, 
we conclude that
under any given bound of the primary distortion,
there are at most finitely many distinct nondegenerate gluings 
up to equivalence.
\end{proof}

\subsection{Distortion at atoroidal vertices --- a remark}\label{Subsec-atorVertexDistortion}
The distortions at the atoroidal vertices are not used for proving Proposition \ref{finiteGluings}.
The reason is explained by the following lemma.

\begin{lemma}\label{atorVertexDistortion} 
Let $(\Lambda,\pieces)$ be a preglue graph of geometrics, and
$v\in\vrtx(\Lambda)$ be a vertex of valence $n_v$
which corresponds to an atoroidal piece $J_v\subset\pieces$. 
Then for any gluing
$\phi\in\Phi(\Lambda,\pieces)$,
the following comparison holds
 for some constant $C>0$ which depending only on the topology of $J_v$:
$$\distortion_v(\phi)\leq C\cdot\left(\prod_{\delta\in\widetilde\edge(v)}\distortion_{e(\delta)}(\phi)\right)^{\frac{2}{n_v}}.$$
Here $\widetilde\edge(v)$ stands for the edge ends adjacent to $v$, 
and $e(\delta)$ stands for the edge that owns the end $\delta$.
\end{lemma}

\begin{proof} For simplicity we rewrite $J_v$ as $J$, and $n_v$ as $n$. 
Rewrite the submodule $\partial_*H_2(J,\partial J;\ZZ)$ of $H_1(\partial J;\ZZ)$ as $W$, 
and the subspace $\partial_*H_2(J,\partial J;\RR)$ of $H_1(\partial J;\RR)$ as $W_\RR$.
By definition we have $\qf_\phi\geq\qf_J$,
both positive definite on $H_1(\partial J;\RR)$, (see Subsection \ref{Subsec-qfGluings}).
So the unit ball $B_\phi$ of $\qf_\phi$ is contained 
by the unit ball $B_J$ of $\qf_J$.
It suffices to show for some constant $C_0>0$ independent of $\phi$,
$$\Delta(W,\qf_\phi)\leq C_0\cdot\Delta(H_1(\partial J;\ZZ),\qf_\phi).$$
By choosing a basis of $H_1(\partial J;\ZZ)$ as
an orthonormal basis, 
we fix a reference inner product on $H_1(\partial J;\RR)$.
Denote by $\mu_{2n}$ the induced $2n$--dimensional volume measure on $H_1(\partial J;\RR)$,
and by $\mu_n$ the induced $n$--dimensional volume measures on $W_\RR$ and
on $W^\perp_\RR$. 
It suffices to show for some constant $C_1>0$ independent of $\phi$,
$$\mu_{2n}(B_\phi)\leq C_1\cdot\mu_{n}(W_\RR\cap B_\phi).$$
Observe
$$\mu_{2n}(B_\phi)=\frac{\omega_{2n}}{\omega_n^2}\cdot\mu_{n}(W_\RR\cap B_\phi)\cdot\mu_{n}(\bar{B}_\phi),$$
where $\omega_m$ stands the volume of an $m$--dimensional Euclidean unit ball, 
and where $\bar{B}_\phi$ stands for the image of $B_\phi$ in $W^\perp$ under orthogonal projection. 
Then the above inequality about $\mu_{2n}(B_\phi)$ 
follows immediately from the comparison
$$\mu_{n}(\bar{B}_\phi)\leq \mu_{n}(\bar{B}_J),$$
where $\bar{B}_J$ stands for the image $B_J$ in $W^\perp$
under orthogonal projection. 
In fact, the uniform constant $C_1>0$
can be taken to be $\omega_{2n}/\omega_n^2$ times $\mu_n(\bar{B}_J)$.
\end{proof}

\section{Domination onto non-geometric $3$--manifolds}\label{Sec-nonGeomCase}
In this section, we bound the primary distortion of gluings assuming domination.
This is the content of Proposition \ref{distortionDomination} below.
Theorem \ref{main-dominate} is an immediate consequence
of Propositions \ref{finiteGluings} and \ref{distortionDomination}, 
since equivalent gluings yield homeomorphic $3$--manifolds 
by Definition \ref{eqvGluing}.
We also summarize the proof of Theorem \ref{main-dominate}
at the end of this section.

\begin{proposition}\label{distortionDomination} 
Let $M$ be an orientable closed $3$--manifold.
Suppose that $N_\phi$ is an orientable closed irreducible $3$--manifold
obtained from a nondegenerate gluing $\phi\in\Phi(\Lambda,\pieces)$ 
of a preglue graph of geometrics $(\Lambda,\pieces)$. 
Then there exists some constant $C=C(M)>0$ with the following property:
If $M$ dominates $N_\phi$, then 
the primary distortion satisfies $\distortion_\Lambda(\phi)<C$.\end{proposition}

We prove Proposition \ref{distortionDomination} in the rest of this section. 
In fact, we show that it suffices to assume that the underlying graph
is loopless and entire (Subsection \ref{Subsec-reduction}).
Then we present the proof under that additional assumption.
This trick of reduction helps to simplify the notations in a number of places,
(Subsection \ref{Subsec-looplessGraphCase}).

\subsection{Reduction to loopless entire graphs}\label{Subsec-reduction} 
We say that a graph $\Lambda$ is \emph{loopless}
if it contains no loop edge. 
We say that a graph is \emph{entire} 
if it contains no semi edges or semi vertices.

\begin{lemma}\label{reductionLooplessGraph} 
The statement of Proposition \ref{distortionDomination} holds true
if it holds true
under the additional assumption that $\Lambda$ is loopless and entire.\end{lemma}

\begin{proof} 
The idea is that $\Lambda$ admits a covering graph $\tilde\Lambda$
(in the orbi-space sense) which has index at most $4$ 
and which is loopless and entire.
To be precise, 
suppose that $f\colon M\to N_\phi$ is a map of nonzero degree.
We rewrite $N_\phi$ as $N$ for simplicity. 

Cut $N$ along a maximal disjoint union of incompressible Klein-bottles,
and take two copies $X_0,X_1$ of the resulting compact $3$--manifold.
Glue each component of $\partial X_0$ 
to a unique component of $\partial X_1$ according 
to the gluing pattern of $N$. 
Then we obtain a $2$--fold covering space $\tilde{N}'$ of $N$.
The geometric decomposition of $\tilde{N}'$
gives rise to an entire graph $\tilde\Lambda'$, (possibly disconnected
if $\Lambda$ is already entire). 
Cut $\tilde{N}'$ along the tori which correspond the loop
edges of $\tilde\Lambda'$, 
and glue up two copies of the resulting
compact $3$--manifold, 
according to the gluing pattern of $\tilde{N}'$.
Then we obtain a $2$--fold covering space $\tilde{N}''$ of $\tilde{N}'$.
The graph for $\tilde\Lambda''$ is loopless and entire,
(possibly disconnected if $\tilde\Lambda'$ is already loopless).
Choose a connected component $\tilde{N}$ of $\tilde{N}''$.
Then $\tilde{N}$ covers $N$ of index
at most $4$, 
and it has a loopless entire graph $\tilde\Lambda$.
By Lemma \ref{graphCovering},
we have $\distortion(\tilde{N})=\distortion(N)$.
However, 
$\tilde{N}$ is dominated by a (connected) covering
$\tilde{M}$ of $M$ with index at most $4$, 
so $\distortion(\tilde{N})$ is at most some constant $c(\tilde{M})>0$, 
guaranteed by the assumption. 
Note that there are only finitely many such $\tilde{M}$ up to isomorphism,
because $\pi_1(M)$ is finitely generated. 
Take $C>0$ to be the maximum for all the constants
$c(\tilde{M})$ and for all the coverings $\tilde{M}$ of $M$ 
of index at most $4$.
Then we have $\distortion_\Lambda(\phi)=\distortion(N)<C$,
since $\phi\in\Phi(\Lambda,\pieces)$ is a nondegenerate gluing.
\end{proof}

\subsection{The loopless entire graph case}\label{Subsec-looplessGraphCase}
We prove Proposition \ref{distortionDomination} 
for loopless entire graphs:

\begin{proposition}\label{looplessGraphCase} 
Proposition \ref{distortionDomination} holds true under 
the additional assumption that $\Lambda$ is loopless and entire.\end{proposition}

We prove Proposition \ref{looplessGraphCase} in the rest of this subsection. 
In fact, 
we show that under the assumption of Proposition \ref{looplessGraphCase},
$\distortion_v(\phi)<C$ holds for any vertex $v\in\vrtx(\Lambda)$, 
where the constant $C>0$ 
can be chosen depending only on the triangulation number $t(M)$ of $M$.
(The \emph{triangulation number} of a closed $3$--manifold refers to
the smallest possible number of $3$--simplices in a simplicial decomposition of $M$.)
Similarly, $\distortion_e(\phi)<C$ holds for any edge $e\in\edge(\Lambda)$. 
The strategy is similar to \cite[Section 3]{AL}.
It is actually possible to bound the distortions
using the presentation length of $\pi_1(M)$ 
as introduced in \cite[Definition 3.1]{AL}.
However, to establish a factorization theorem like \cite[Theorems 3.2 and 3.6]{AL}
would seem unnecessarily involved for the goal of the present paper.
We appeal to the Poincar\'{e}--Lefschetz duality for a substitute.

\subsubsection{Partial geometrization of the target manifold}\label{Subsubsec-partialGeometrization}
Let $(\Lambda,\pieces)$ be a preglue graph of geometrics whose graph $\Lambda$ is loopless and entire,
and $\phi\in\Phi(\Lambda,\pieces)$ a nondegenerate gluing.
Suppose that $\Lambda$ is nontrivial,
then the glued-up $3$--manifold $N_\phi$ has a nontrivial geometric decomposition 
precisely as described by $(\Lambda,\pieces)$.
We would like to furnish $N_\phi$ 
with certain auxiliary Riemannian metric that is closely related
to the geometrization of the pieces,
so as to make effective estimates.

Given any $(\Lambda,\pieces)$ and $\phi$ as above, 
we simply rewrite $N_\phi$ by $N$.
Denote by
$$\tori=\bigsqcup_{e\in\edge(\Lambda)}T_e\subset N,$$
the union of cutting tori of $N$ in its geometric decomposition,
and by
$$\mathcal{U}=\bigsqcup_{e\in\edge(\Lambda)}\mathcal{U}_e\subset N,$$ 
a compact collar neighborhood of $\tori$.
The neighborhood boundary $\partial\mathcal{U}$ 
can be naturally identified
(up to isotopy) as the disjoint union of tori 
$T_\delta$, 
for all $\delta\in\widetilde\edge(\Lambda)$.
The complement in $N$ of the interior 
of $\mathcal{U}$ 
can be naturally identified with the disjoint union of 
the geometric pieces $\pieces$. 
We fix such an identification 
for the rest of this subsection, 
so we can write
$$N=\pieces\cup_{\partial\mathcal{U}}\mathcal{U}.$$

For any proper Margulis number $\epsilon>0$ for $\Hyp$ (hence also for $\mathbb{H}^2$),
we say that a Riemannian metric $\rho$ on $N$
is \emph{$\epsilon$--partially geometric}
if the following conditions are satisfied
for all vertices $v\in\vrtx(\Lambda)$:
\begin{itemize}
\item If the corresponding piece $J_v$ is atoroidal,
$(J_v,\rho)$ is isometric to the corresponding
complete hyperbolic $3$--manifold $J^\geo_v$ 
with the open $\epsilon$--thin horocusps removed.
\item If the corresponding piece $J_v$ is Seifert fibered,
$(J_v,\rho)$ is isometric to a corresponding complete $\HypEuc$--geometric $3$--manifold $J^\geo_v$ 
with the open horizontally $\epsilon$--thin horocusps
removed.
\end{itemize} 
Here a \emph{horizontally $\epsilon$--thin horocusp} refers to the preimage
in $J^\geo_v$ of an $\epsilon$--thin horocusp of the base $2$--orbifold
$\mathcal{O}^\geo$.
Note that the conditions determine $\rho$ uniquely on 
any atoroidal $J_v$ because of the Mostow rigidity,
or uniquely up to rescaling along the fiber direction
on any Seifert fibered $J_v$.
We impose no restriction to $\rho$ on the rest part $\mathcal{U}$.

When working with partially geometric Riemannian metrics,
we are supposed to focus on the hyperbolic geometry of the atoroidal pieces
and the horizontal hyperbolic geometry of the Seifert fibered pieces.
In particular, we introduce the following special notion of area,
which counts only the part of interest to us.
Let $j\colon K\to N$ be any piecewise smooth immersion
of a simplicial $2$--complex $K$ into $N$.
Suppose that $N$ is endowed with some 
Riemannian metric $\rho$ that is $\epsilon$--partially geometric.
For each hyperbolic piece $J_v$, there is an induced area measure
on $j^{-1}(J_v)\subset K$, which pulls back the area measure on $J_v$.
For each $\HypEuc$--geometric piece
$J_v$, there is an induced area measure on $j^{-1}(J_v)\subset K$,
which pulls back the horizontal area measure on $J_v$,
or in other words,
which pulls back the area measure on the hyperbolic base $2$--orbifold.
(Intuitively it measures the area after projecting onto the base
$2$--orbifold; see \cite[Subsection 3.3]{AL} for details.)
Then we define the \emph{effective area} of $(K,j)$
to be sum of the area measures of 
$j^{-1}(J_v)\subset K$ for all $v\in\vrtx(\Lambda)$,
denoted as $\Area^{\mathtt{eff}}_\rho(K,j)\geq0$.
Below we actually adopt a simpler notation
$$\Area(j(K))\geq0$$
and refer to $\Area(j(K))$ simply as the \emph{area} of $j(K)$,
since no other notions of area are used in this paper.
The area of a simplicial $2$--chain of $K$ is the sum
of the areas of the $2$--simplices weighted by 
the absolute values of their coefficients. 

Suppose that $\epsilon^*>0$ is a proper Margulis number for $\Hyp$.
For any $\epsilon$--partially geometric Riemannian metric on $N$ 
with $0<\epsilon<\epsilon^*$,
we introduce some geometrically defined collar neighborhoods
 $\mathcal{W}_e$ of $\mathcal{U}_e$ and 
$\mathcal{W}_v$ of $J_v$ as follows.
(Presumably $\epsilon*\approx \epsilon_3$ 
the Margulis constant of $\Hyp$ and $\epsilon\approx 0$.)
For each edge $e\in\edge(\Lambda)$, denote by 
$$\mathcal{W}_e=\mathcal{W}_e(\epsilon^*)\subset N$$ 
the union of $\mathcal{U}_e$ and the compact
$\epsilon^*$--thin (or horizontally $\epsilon^*$--thin)
horocusp neighborhoods of its adjacent pieces.
(There are two distinct adjacent pieces as $\Lambda$ is loopless and entire.)
As $\Lambda$ is loopless, each $\mathcal{W}_e$ deformation-retracts to $T_e$,
so there is a quadratic form on the subspace $\partial_*H_2(\mathcal{W}_e,\partial\mathcal{W}_e;\RR)$
of $H_1(\partial\mathcal{W}_e;\RR)$, naturally induced from $\qf_\phi$ on
$H_1(T_\delta;\RR)\oplus H_1(T_{\bar\delta};\RR)$, 
where $\delta,\bar\delta$ stand for the two ends of $e$.
For each vertex $v\in\vrtx(\Lambda)$, denote by
$$\mathcal{W}_v=\mathcal{W}_v(\epsilon^*)\subset N$$
the union of $J_v$ and all the $\mathcal{W}_e$ where $e$ runs over the edges adjacent to $v$. 
As $\Lambda$ is loopless, each $\mathcal{W}_v$
deformation-retracts onto $J_v$,
so there is a quadratic form on the subspace
$\partial_*H_2(\mathcal{W}_v,\partial\mathcal{W}_v;
\RR)$ of $H_1(\partial\mathcal{W}_v;\RR)$, naturally induced from $\qf_\phi$ on
$H_1(\partial J_v;\RR)$.
See Figure \ref{figNhdW} for an illustration.

\begin{figure}[htb]
\centering
\def\svgwidth{\columnwidth}
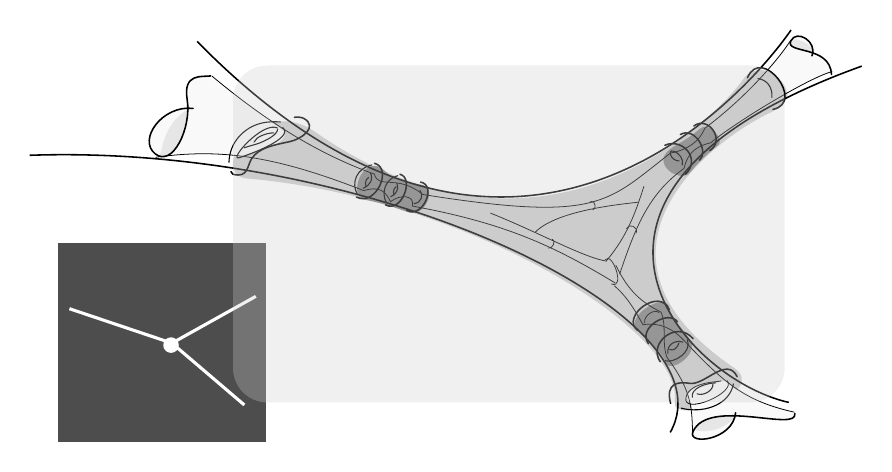
\caption{A cartoon of $N$ near a geometric piece with $3$ adjacent tori}\label{figNhdW}
\end{figure}

The following isoperimetric comparison
serves as the junction point between the geometric aspect and 
the quadratic forms associated to gluings:

\begin{lemma}\label{isoperimetric} 
Let $\epsilon^*>0$ be a proper Margulis number for $\Hyp$.
Suppose that $N$ is furnished with an $\epsilon$--partially geometric Riemannian metric
with $0<\epsilon<\epsilon^*$.
For any vertex $v\in\vrtx(\Lambda)$, 
suppose that $j\colon (S,\partial S)\to (\mathcal{W}_v,\partial\mathcal{W}_v)$ 
is a piecewise smooth proper immersion of an oriented compact surface.
Then the following estimate holds true:
$$\Area(j(S))\geq 4\cdot\left(\sinh\left(\frac{\epsilon^*}2\right)-\sinh\left(\frac{\epsilon}2\right)\right)\cdot\sqrt{\qf_\phi\left(j_*[\partial S]\right)},$$
where $j_*[\partial S]$ lies in $\partial_*H_2(\mathcal{W}_v,\partial\mathcal{W}_v;\ZZ)$. 
The same statement holds true for any edge $e\in\edge(\Lambda)$ in place of $v$.\end{lemma}

\begin{remark}\label{isoperimetricRemark} 
By hyperbolic geometry,
$4\sinh(\epsilon/2)$ is the Euclidean length of the shortest geodesic 
on the boundary of a hyperbolic horocusp 
whose pointwise injectivity radii are at most $\epsilon$
(maximized on the boundary).
Note that
the statement of Lemma \ref{isoperimetric} also makes sense when $\epsilon^*$ equals the Margulis constant $\epsilon_3$ for $\Hyp$,
and it holds true since we can take $\epsilon^*$ approaching $\epsilon_3$ from below.
The right-hand side of the inequality can be strengthened to be 
$\left(1-4\sinh(\epsilon/2)\right)\cdot\sqrt{\qf_\phi(j_*[\partial S])}$.
In fact, if we consider mutually disjoint maximal horocusps,
instead of the Margulis horocusps used in $\mathcal{W}_e$,
the length of the shortest geodesic on $\partial{W}_e$
will be replaced by $1$, (see Adams \cite{Ad}). 
This nice improvement is pointed out to the author by I. Agol.
\end{remark}

\begin{proof} 
We only prove the vertex case. The edge case is completely similar.
Let $v\in\vrtx(\Lambda)$ be a vertex. 
Denote by $\edge(v)$ the edges adjacent to $v$.  
As $\Lambda$ is loopless and entire,
$e\in\edge(v)$ has two ends, which we denote as $\widetilde\edge(e)=\{\delta,\bar\delta\}$.
The corresponding components
of $\mathcal{W}_e\setminus\mathring{\mathcal{U}}_e$ 
are denoted as $\mathcal{W}_\delta,\mathcal{W}_{\bar\delta}$.
Suppose 
$$j_*[\partial S]=\sum_{e\in\edge(v)}\alpha_e,$$
according to the direct-sum decomposition
$$H_1(\partial\mathcal{W}_v;\RR)\cong\bigoplus_{e\in\edge(v)}\,H_1(T_e;\RR).$$
By a well-known calibration argument,
the (effective) area of 
$j(S)\cap\mathcal{W}_\delta$ is at least 
$4(\sinh(\epsilon^*/2)-\sinh(\epsilon/2))\cdot\sqrt{\qf_{J_\delta}(\alpha_e)}$,
for any $\delta\in\widetilde\edge(e)$ and any $e\in\edge(v)$.
Here $\delta$ stands for the end of $e$ pointing out from $v$,
and $J_\delta\subset\pieces$ stands for the corresponding adjacent piece of $J_v$.
(See Subsection \ref{Subsec-qfGluings} for the definition of $\qf_{J_\delta}$). 
We estimate
\begin{eqnarray*}
\Area(j(S))&\geq&\sum_{e\in\edge(v)}\sum_{\delta\in\widetilde\edge(e)}4\cdot\left(\sinh\left(\frac{\epsilon^*}2\right)-\sinh\left(\frac{\epsilon}2\right)\right)\cdot\sqrt{\qf_{J_\delta}(\alpha_e)}\\
&\geq&4\cdot\left(\sinh\left(\frac{\epsilon^*}2\right)-\sinh\left(\frac{\epsilon}2\right)\right)\cdot\sqrt{\sum_{e\in\edge(v)}\sum_{\delta\in\widetilde\edge(e)}\qf_{J_\delta}(\alpha_e)}\\
&=&4\cdot\left(\sinh\left(\frac{\epsilon^*}2\right)-\sinh\left(\frac{\epsilon}2\right)\right)\cdot\sqrt{\sum_{e\in\edge(v)}\qf_\phi(\alpha_e)}\\
&=&4\cdot\left(\sinh\left(\frac{\epsilon^*}2\right)-\sinh\left(\frac{\epsilon}2\right)\right)\cdot\sqrt{\qf_\phi(j_*[\partial S])},
\end{eqnarray*}
which is as desired.
\end{proof}

We observe another elementary estimate of linear algebra to be used later.

\begin{lemma}\label{largeGenerator} 
Let $\epsilon^*>0$ be a proper Margulis number for $\Hyp$.
Suppose that $N$ is furnished with an $\epsilon$--partially geometric Riemannian metric
with $0<\epsilon<\epsilon^*$.
For any vertex $v\in\vrtx(\Lambda)$, 
suppose that $\alpha_1,\cdots,\alpha_m$ are
elements of $\partial_* H_2(\mathcal{W}_v,\partial\mathcal{W}_v;\ZZ)$ 
which span $\partial_* H_2(\mathcal{W}_v,\partial\mathcal{W}_v;\RR)$ over $\RR$.
Then the following inequality holds for some $\alpha_k$, $1\leq k\leq m$:
$$\sqrt{\qf_\phi(\alpha_k)}\geq \distortion_v(\phi).$$
The same statement holds true for any edge $e\in\edge(\Lambda)$ in place of $v$.
\end{lemma}

\begin{proof} This follows from a Minkowski-type inequality for lattices. 
Without loss of generality, we assume that $m$ is minimal, 
so $m$ equals to the valence of $v$. 
Consider the volume of the parallelogram
spanned by the elements $\alpha_i$ with respect to 
the inner product induced by $\qf_\phi$, on $\partial_* H_2(\mathcal{W}_v,
\partial\mathcal{W}_v;\RR)$.
Then it follows
$$\prod_{i=1}^{n_v}\sqrt{\qf_\phi(\alpha_i)}\geq 
\left|\det(\alpha_1,\cdots,\alpha_{n_v})\right|\cdot\sqrt{\Delta(\partial_* H_2(\mathcal{W}_v,
\partial\mathcal{W}_v;\ZZ),\qf_\phi)},$$
where $\det(\alpha_1,\cdots,\alpha_{n_v})$ is the determinant
of $(\alpha_1,\cdots,\alpha_{n_v})$, regarded as a square matrix over any basis of 
$\partial_* H_2(\mathcal{W}_v,\partial\mathcal{W}_v;\ZZ)$.
Note that $|\det(\alpha_1,\cdots,\alpha_{n_v})|\geq1$
because the determinant is a nonzero integer.
By the definition of vertex distortion,
we estimate
$$\prod_{i=1}^{n_v}\sqrt{\qf_\phi(\alpha_i)}
\geq\sqrt{\Delta(\partial_* H_2(\mathcal{W}_v,
\partial\mathcal{W}_v;\ZZ),\qf_\phi)}=\distortion_v(\phi)^{n_v},$$
so the asserted inquality follows immediately.

The edge case is completely similar.
\end{proof}

\subsubsection{Efficient normal positioning of the $2$--skeleton}
Given any map $f\colon M\to N$,
(presumably of nonzero degree,)
we would like to modify $f$ by homotopy 
so that it sits in nice position with respect
to a simplicial decomposition of $M$ and 
the geometric decomposition of $N$.
More specifically, we would like the $2$--skeleton of $M$ under $f$ to be 
piecewise smoothly immersed, 
and in certain normal position with respect to $\mathcal{U}$,
and of certain effectively bounded area 
according to the geometrization of $\pieces$.
These points are made precise in Lemma \ref{pullStraight},
but we need some \textit{ad hoc} terms to state it.

In hyperbolic plane or $3$--space if we intersect a geodesic $2$--simplex generically with
a compact convex domain with smooth boundary, 
the resulting region will be a disk 
with only finitely many possible combinatorial boundary configurations.
This geometric picture inspires us with the following topological notion.
A \emph{handle} of a $2$--simplex is defined 
to be an embedded (compact) disk that intersects each edge
either in a compact interval or in the empty set.
For example, if a handle contains no vertex of the $2$--simplex,
it is either an \emph{isolated disk} (a disk in the open $2$--simplex bounded by a simple closed curve),
or a \emph{half disk} (a bigonal disk bounded by a properly embedded arc 
and a compact interval in an open edge),
or a \emph{normal band} (a rectangular disk bounded by two parallel properly embedded arcs
and two compact intervals in distinct open edges),
or a \emph{monkey handle} 
(a hexagonal disk bounded by three mutually nonparallel properly embedded arcs
and three compact intervals in mutually distinct open edges).
There are another four possible configurations if the handle contains vertices,
which we sometimes refer to as \emph{cornered handles}.

A \emph{handle complex} in a simplicial $2$--complex refers to 
a union of handles in the $2$--simplices 
which are mutually disjoint in any $2$--simplex and which
fit exactly along the $1$--skeleton.
The latter condition means that the intersections of an edge
with any two handles are either disjoint or the same.
A handle complex in a simplicial $2$--complex is said to be \emph{complete}
if every $2$--simplex intersects the handle complex 
in a (possibly empty) union of mutually disjoint handles.
The \emph{frontier} of a complete handle complex hence refers to the union
of all the boundary simple closed curves 
and all the boundary arcs of the handles 
which are not compact intervals of edges.


\begin{lemma}\label{pullStraight} 
Adopt the assumptions and notations for $N$ as in Section \ref{Subsubsec-partialGeometrization}.
Let $M$ be an orientable closed $3$--manifold furnished with a simplicial decomposition.
Let $\beta>0$ be an arbitrary constant and $\epsilon^*>0$ be a proper Margulis number for $\Hyp$.
Then for any (continuous) map $f\colon M\to N$,
there exists an $\epsilon$--partially geometric Riemannian metric on $N$
for some constant $0<\epsilon<\epsilon^*$,
and $f$ can be modified by homotopy to satisfy all the following conditions:
\begin{itemize}
\item 
For any vertex $v\in\vrtx(\Lambda)$,
the intersections $M^{(2)}\cap f^{-1}(\mathcal{W}_v)$
is a complete handle complex in the $2$--skeleton $M^{(2)}$
with frontier $M^{(2)}\cap f^{-1}(\partial \mathcal{W}_v)$.
The same statement holds true for any edge $e\in\edge(\Lambda)$.
\item
The following estimate about the effective area of the $2$--skeleton holds:
$$\Area\left(f\left(M^{(2)}\right)\right)\,<\,\#\left\{\textrm{2-simplices of }M^{(2)}\right\}\times\pi+\beta.$$
\end{itemize}
\end{lemma}

The reader may think of $\beta>0$ to be very small, so the area upper bound is roughly $\pi$
times the number of $2$--simplices in the $2$--skeleton of $M$.
Intuitively, we first want to arrange the original $f$ by homotopy, so that it intersects
$\mathcal{U}$ in certain minimal generic position, 
then we want to pull $f$ on the complement of $f^{-1}(\mathcal{U})$,
so that it becomes somehow straight in the geometric pieces.
The proof below actually follow this idea,
but we present with sufficient details for the reader's reference.

\begin{proof}
Since $N$ is an Eilenberg--MacLane space, 
any map $f'|_{M^{(2)}}\colon M^{(2)}\to N$ 
that is homotopic to the restriction of $f$ to 
the $2$--skeleton $M^{(2)}$
can be extended to be a map $f'\colon M\to N$,
and any such extension $f'$ is homotopic to $f$.
Therefore, it suffices to modify the restricted map
$f|_{M^{(2)}}\colon M^{(2)}\to N$ by homotopy to satisfy the asserted conditions.
Below we rewrite $f|_{M^{(2)}}$ as $f_0$ for simplicity.

We first modify $f_0$ by homotopy
to obtain a piecewise smooth map $f_1\colon M^{(2)}\to N$,
such that $f_1^{-1}(\mathcal{U})$ is a complete handle complex in $M^{(2)}$ formed by normal bands,
and that $f_1^{-1}(\partial\mathcal{U})$ is the frontier of $f_1^{-1}(\mathcal{U})$.
This follows from a familiar argument in normal surface theory, which we sketch as follows.
Possibly after a small perturbation of $f_0$, 
we may assume it to be piecewise smoothness and transverse to $\tori$.
Take a homotopy modification $f'_0$ of $f_0$ which minimizes the cardinality of $f_0^{-1}(\tori)\cap M^{(1)}$.
Then $(f'_0)^{-1}(\tori)$ must only intersect any $2$--simplex of $M^{(2)}$ 
in a possibly empty union of mutually disjoint properly embedded normal arcs and simple closed curves.
Here a normal arc in a $2$--simplex refers to a properly embedded arc
whose endpoints lie in a pair of distinct open edges.
In fact, any non-normal properly embedded arc would bound
a disk in the $2$--simplex with some subarc of an edge,
so its pair endpoints would be removable from $(f'_0)^{-1}(\tori)\cap M^{(1)}$ by Whitney's disk trick,
which is impossible due to the minimality assumption on $f'_0$.
We modify $f'_0\colon M^{(2)}\to N$ furthermore on the interior of $2$--simplices,
so that the resulting map $f''_0\colon M^{(2)}\to N$ has the property that
$(f''_0)^{-1}(\tori)$ contains no simple closed curves in open $2$--simplices.
Note that $f''_0$ is automatically homotopic to $f'_0$, because $N$ is an Eilenberg--MacLane space.
To obtain $f''_0$, 
the modification can be done by induction 
starting from any inner most simple closed curve of $(f'_0)^{-1}(\tori)$ contained in an open $2$--simplex,
using the incompressiblity of the components of $\tori$ in $N$.
In the end, we obtain a map $f''_0\colon M^{(2)}\to N$ homotopic to $f_0$,
piecewise smooth, transverse to $\tori$, and 
such that $(f''_0)^{-1}(\tori)$ is a union of normal arcs in $2$--simplices.
Possibly after composing $f''_0$ with an isotopically trivial homeomorphism of $N$,
we obtain a homotopy modification $f_1$ of $f_0$ as claimed.

Denote by $M^{(2)}_{\mathcal{U}}$ the complete handle complex $f_1^{-1}(\mathcal{U})$ in $M^{(2)}$
formed by normal bands.
Let $\epsilon^*>0$ be any given proper Margulis number for $\Hyp$.
Suppose that $N$ is endowed with an $\epsilon$--partially geometric Riemannian metric.
At this point, we only require $0<\epsilon<\epsilon^*$.
We describe an explicit recipe 
to modify $f_1$ on the complement of $M^{(2)}_{\mathcal{U}}$.
The resulting map $f_2\colon M^{(2)}\to N$
is homotopic to $f_1$ relative to $M^{(2)}_{\mathcal{U}}$.
It satisfies the asserted conditions of Lemma \ref{pullStraight}
when $\epsilon$ is sufficiently small.

The recipe for constructing $f_2$ is as follows. 
Denote by $M^{(2)}_{J_v}$ the complete handle complex $f_1^{-1}(J_v)$ in $M^{(2)}$,
for any vertex $v$.
Denote by $M^{(2)}_{T_\delta}$ the union of frontier components $f_1^{-1}(T_\delta)$
of $M^{(2)}_{J_v}$, for any end of edge $\delta$ adjacent to $v$.
For convenience of the description, 
we fix a map $\phi_v\colon J^{\geo}_v\to N$ that sends $J^{\geo}_v$ homeomorphically 
onto the component of $N\setminus\tori$ which contains $J_v$,
such that $(\phi_v)^{-1}(J_v)$ is sent isometrically onto $J_v$.
Via $\phi_v$ we may identify the restriction of $f_1$ to $M^{(2)}_{J_v}$
as a map $\tilde{f}_{1,v}\colon M^{(2)}_{J_v}\to J^{\geo}_v$.
So for any boundary torus $T_\delta$ of $J_v$,
the map $\tilde{f}_{1,v}$ sends $M^{(2)}_{T_\delta}$ to
the $\epsilon$--horocusp boundary $\phi^{-1}_v(T_\delta)$ of $J^{\geo}_v$.
For any handle $R$ of $M^{(2)}_{J_v}$, we take a subdivision of $R$
into bigon regions and triangular regions as follows.
Since $R$ cannot contain isolated disks, 
$\partial R$ meets the frontier of $M^{(2)}_{J_v}$
in at most three mutually disjoint arcs (possibly none),
we call these the \emph{frontier edges} of $R$.
We call the compact intervals that appear in $\partial R\cap M^{(1)}$ 
the \emph{simplicial edges} of $R$,
and call the vertices that appear in $\partial R\cap M^{(0)}$
the \emph{simplicial vertices} of $R$.
For each frontier edge of $R$,
we construct a properly embedded arc in $R$ which 
bounds a bigon region with that frontier edge.
We further construct diagonal arcs of $R$ to divide
the rest part of all the bigonal regions in $R$
into triangular regions.
If $J_v$ is atoroidal, we modify $\tilde{f}_{1,v}$ on $R$
by homotopy relative to the frontier edges,
so that the resulting map $\tilde{f}'_{1,v}|_R\colon R\to J^{\geo}_v$
sends any simplicial vertex of $R$ to the $\epsilon^*$--thick part of $J^{\geo}_v$,
and sends any simplicial edge of $R$ and the constructed arcs of $R$
to geodesic segments in $J^{\geo}_v$, 
and sends the bigonal or triangular regions to $J^{\geo}_v$
minimizing the area subject to the above requirements.
In this way, the triangular regions are necessarily locally geodesically immersed,
while the bigonal regions are contained in the horocusps bounded by $\phi_v^{-1}(\partial J_v)$.
If $J_v$ is Seifert fibered, the construction for the modification 
$\tilde{f}'_{1,v}$ on $R$ is similar as the atoroidal case. 
The triangular regions are locally geodesically immersed
after projecting to the hyperbolic base $2$--orbifolds.
The bigonal regions are again contained in the horocusps bounded by $\phi_v^{-1}(\partial J_v)$.
To finish the construction,
we observe that the modified maps $\tilde{f}'_{1,v}|_R$ for all $R$ 
fit naturally, so we obtain a map $\tilde{f}'_{1,v}\colon M^{(2)}_{J_v}\to J^{\geo}_v$.
We define $f_2\colon M^{(2)}\to N$ by $\phi_v\circ\tilde{f}'_{1,v}$ on any $M^{(2)}_{J_v}$,
and by $f_1$ on $M^{(2)}_{\mathcal{U}}$.

It is clear from the construction that $f_2$ is homotopic to $f_0$,
and indeed, homotopic to $f_1$ relative to $M^{(2)}_{\mathcal{U}}$.
With respect to $f_2$, every $2$--simplex of $M^{(2)}$ is divided
into the normal bands of $M^{(2)}_{\mathcal{U}}$, 
and the bigonal regions attached to the normal bands,
and the triangular regions attached to the bigonal regions,
and exactly one extra triangular region,
which is confined by the constructed diagonal arcs
or compact intervals of the edges.

For any given constant $\beta>0$, 
the total effective area of the triangular regions attached to bigonal regions
can be bounded by $\beta$ if $\epsilon$ is sufficiently small.
This is because there are only finitely many such triangular regions, 
and moreover,
each of them has an edge whose length or horizontal length is small
with respect to the geometry of $J^\geo_v$ when $\epsilon$ is close to $0$.
It follows that the total effective area of these triangular regions
is small if $\epsilon$ is small.
The only extra triangular region has effective area at most $\pi$
by hyperbolic geometry.
The bigonal regions and the normal bands contributes nothing to the effective area
since they are contained in $f_2^{-1}(\mathcal{U})$.
Therefore the asserted area estimate holds for $f_2$ and for any sufficiently small $\epsilon$.

If the proper Margulis number $\epsilon^*$ is given generic,
the corresponding parts $\mathcal{W}_v$ and $\mathcal{W}_e$
are complete handle complexes in $M^{(2)}$.
Here being generic means that $\partial\mathcal{W}_v$ and $\partial\mathcal{W}_e$
are transverse to the triangular regions.
The key observation here is 
that the triangular regions are totally geodesically or horizontal-geodesically immersed
with respect to the geometry of $J^\geo_v$.
Since the $\epsilon$--horocusps are convex,
it follows that 
any $\mathcal{W}_v$ or $\mathcal{W}_e$ must intersect the triangular regions
in handles of those regions,
when $\epsilon^*$ is generic.
The intersection of $\mathcal{W}_v$ or $\mathcal{W}_e$
with any $2$--simplex of $M^{(2)}$
is the union of some normal bands and their attached bigonal regions,
together with some handles of the triangular regions as above.
So a direct case-by-case argument shows that $\mathcal{W}_v$ or $\mathcal{W}_e$
intersects any $2$--simplex of $M^{(2)}$ in either a handle or an empty set,
and hence $\mathcal{W}_v$ and $\mathcal{W}_e$ are complete handle complexes
in $M^{(2)}$, when $\epsilon^*$ is generic.

Even if the given proper Margulis number $\epsilon^*$ is non-generic,
we may still perturbe $f_2$ slightly without affecting the area estimate,
then the argument for the generic case works.
Therefore, the asserted complete handle complex structure of $\mathcal{W}_v$ and $\mathcal{W}_e$
holds for $f_2$ possibly after some small perturbation.

We conclude that for any constant $\beta>0$ and proper Margulis constant $\epsilon^*>0$,
there exists
an $\epsilon$--partially geometric Riemannian metric on $N$ for some $0<\epsilon<\epsilon^*$,
and moreover, 
some homotopy modification of $f_0$ satisfies the asserted properties.
In fact, the homotopy modification can be chosen as a small perturbation of the above $f_2$,
and the properties hold for any sufficiently small $\epsilon$.
We also point out that, as indicated by the construction,
an upper bound for the sufficient smallness of $\epsilon$ 
depends on the simplicial decomposition of $M$ and 
the homotopy class of $f|_{M^{(2)}}\colon M^{(2)}\to N$,
besides the constants $\beta$ and $\epsilon^*$.
\end{proof}

If $f\colon M\to N$ is a map that satisfies the listed conditions of Lemma \ref{pullStraight},
we adopt the following notations.
For any vertex $v\in\vrtx(\Lambda)$, 
denote by  $M^{(2)}_{\mathcal{W}_v}$
the complete handle complex $f^{-1}(\mathcal{W}_v)\cap M^{(2)}$
in $M^{(2)}$,
and by 
$M^{(2)}_{\partial\mathcal{W}_v}$
the frontier $f^{-1}(\partial\mathcal{W}_v)\cap M^{(2)}$ of $M^{(2)}_{\mathcal{W}_v}$.
For any edge $e\in\vrtx(\Lambda)$, 
denote by $M^{(2)}_{\partial\mathcal{W}_e}$ and by $M^{(2)}_{\partial\mathcal{W}_e}$
similarly.

The following bounded spanning lemma is similar to \cite[Lemma 3.4]{AL}.
It supplies as crucial an ingredient for our situation.
We show that the second real relative homologies
of $M^{(2)}_{\mathcal{W}_v}$ and $M^{(2)}_{\mathcal{W}_e}$
are all spanned by finite subsets of uniformly bounded area:

\begin{lemma}\label{generatingSet} 
Let $\beta>0$ be a constant and $\epsilon^*>0$ be a proper Margulis number for $\Hyp$.
With the assumptions and notations of Lemma \ref{pullStraight},
suppose that $N$ is endowed with some $\epsilon$--partially geometric Riemannian metric
and that $f\colon M\to N$ satisfies the listed conditions of Lemma \ref{pullStraight}.
Denote by $t$ the number of $3$--simplices in the simplicial decomposition of $M$.

Then there exists some constant $A(\beta,t)>0$ 
with the following properties:
\begin{itemize}
\item For any vertex $v\in\vrtx(\Lambda)$, 
the relative homology 
$H_2(M^{(2)}_{\mathcal{W}_v}, M^{(2)}_{\partial\mathcal{W}_v};\RR)$
is spanned over $\RR$ by some finite subset
represented by some relative $\ZZ$--cycles 
of area individually bounded by $A(\beta,t)$.
\item
The same statement holds true for any edge $e\in\edge(\Lambda)$ 
in place of $v$.
\end{itemize}
\end{lemma}

\begin{remark}\label{A_expression}
The constant $A(\beta,t)$ can be taken explicitly as $3^{4t+2}(2t+1)(2t\pi+\beta)$,
as indicated by our proof below.
This bound is slightly better than 
$3^{6t}(36t^2+8t)\pi$, which is given in \cite[Lemma 3.4]{AL},
but apparently both of the bounds are far from optimal.
\end{remark}

The proof of \cite[Lemma 3.4]{AL} actually works for our current situation 
after some adaptation. 
However, we provide a self-contained proof with slightly different exposition,
for the reader's reference.

\begin{proof}
We only argue for the vertex case. The edge case is completely similar.

Note that $M^{(2)}_{\mathcal{W}_v}$ is a complete handle complex in $M^{(2)}$
with frontier $M^{(2)}_{\partial\mathcal{W}_v}$. 
We regard the handle complex structure 
as a cell $2$--complex decomposition of $M^{(2)}$,
where the open $2$--cells are precisely the interior of the handles.
The frontier $M^{(2)}_{\partial\mathcal{W}_v}$ is a cell $1$--subcomplex
and the rest $1$--cells are the components of $M^{(2)}_{\mathcal{W}_v}\cap M^{(1)}$.
For each dimension $i=0,1,2$, denote by
$\mathcal{C}_i(M^{(2)}_{\mathcal{W}_v},M^{(2)}_{\partial\mathcal{W}_v})$,
$\mathcal{Z}_i(M^{(2)}_{\mathcal{W}_v},M^{(2)}_{\partial\mathcal{W}_v})$,
and $\mathcal{B}_i(M^{(2)}_{\mathcal{W}_v},M^{(2)}_{\partial\mathcal{W}_v})$
the $\ZZ$--module of relative cellular $i$--chains, $i$--cycles, and $i$--boundaries,
respectively.
We observe the following natural isomorphism of $\RR$--modules:
$$H_2\left(M^{(2)}_{\mathcal{W}_v},M^{(2)}_{\partial\mathcal{W}_v};\RR\right)
\cong\mathcal{Z}_2\left(M^{(2)}_{\mathcal{W}_v},M^{(2)}_{\partial\mathcal{W}_v}\right)\otimes\RR.$$

Fix an orientation for each cell.
Then $\mathcal{C}_i(M^{(2)}_{\mathcal{W}_v},M^{(2)}_{\partial\mathcal{W}_v})$
is freely generated over the standard unordered basis,
which consists of all the $i$--cells of $M^{(2)}_{\mathcal{W}_v}$ 
that are not contained in $M^{(2)}_{\mathcal{W}_v}$.
Denote by $\|\cdot\|_\infty$ the standard $\ell^\infty$--norm norm on
$\mathcal{C}_i(M^{(2)}_{\mathcal{W}_v},M^{(2)}_{\partial\mathcal{W}_v})\otimes\RR$.
(For any relative cellular $i$--chain $c$,
the norm $\|c\|_\infty\geq0$ is defined to be the maximum absolute value of the coefficients
as $c$ is a unique $\RR$--linear combination over the standard basis.)
Because of Lemma \ref{pullStraight},
the following area estimate holds for any relative cellular $2$--chain
$c\in\mathcal{C}_2(M^{(2)}_{\mathcal{W}_v},M^{(2)}_{\partial\mathcal{W}_v})\otimes\RR$:
$$\Area(c)\leq \|c\|_\infty\times\Area\left(M^{(2)}_{\mathcal{W}_v}\right)< 
\|c\|_\infty\times\Area\left(M^{(2)}\right)<\|c\|_\infty\times (2t\pi+\beta).$$
Note that any simplicial closed $3$--manifold of $t$ $3$--simplices has $2t$ $2$--simplices.

It follows that the statement of Lemma \ref{generatingSet}
is implied by the following purely combinatorial claim:
\begin{itemize}
\item If $X$ is a complete handle complex in a finite simplicial $2$--complex $K$ of $l$ $2$--simplices,
then the finitely generated free $\RR$--module of relative cellular $2$--cycles 
$\mathcal{Z}_2(X,\mathrm{fr}(X))\otimes\RR$ is spanned over $\RR$
by all the $\ZZ$--elements $z\in\mathcal{Z}_2(X,\mathrm{fr}(X))$ 
that satisfy $\|z\|_\infty\leq 3^{2l+2}(l+1)$.
Here $\mathrm{fr}(X)$ stands for the frontier of $X$ in $K$.
\end{itemize}

To prove the combinatorial claim, 
we may assuming that $X$ is connected, 
otherwise working separately with its connected components.
We may also assume that $X$ contains no cornered handles
by the following trick of reduction.
Denote by $\mathcal{N}$ a compact star neighborhood of
the $0$--skeleton $K^{(0)}$, small enough so that
$\mathcal{N}$ only intersects any cornered handle of $X$ in a compact star neighborhood 
of the simplicial vertices bounded by normal arcs.
Then $\hat{X}=X\setminus\mathring{\mathcal{N}}$ is a complete handle complex with frontier
$\mathrm{fr}(\hat{X})=\mathrm{fr}(X)\cup\partial\mathcal{N}$ in $K$,
and there are no cornered handles in $\hat{X}$.
(Here $\mathring{\mathcal{N}}$ stands for the pointset interior of $\mathcal{N}$, which is an open star,
and $\partial{\mathcal{N}}$ stands for the pointset boundary, which is a link.)
The natural isomorphism of $\RR$--modules
$\mathcal{Z}_2(\hat{X},\mathrm{fr}(\hat{X}))\otimes\RR\cong 
H_2(\hat{X},\mathrm{fr}(\hat{X});\RR)\cong H_2(X,\mathrm{fr}(X)\cup \mathcal{N};\RR)\cong
H_2(X,\mathrm{fr}(X);\RR)\cong \mathcal{Z}_2(X,\mathrm{fr}(X))\otimes\RR$
preserves the norm $\|\cdot\|_\infty$ on both ends.
Therefore, 
if the above claim holds for complete handle complexes with no cornered handles,
such as $\hat{X}$, it also holds for any complete handle complex $X$ in general.

Suppose that $X$ is connected and contains no cornered handles.
In this case, the proof of the claim is essentially the content of \cite[Lemma 3.4]{AL}.
The argument below follows that same idea with some slight refinement of the estimates.

Take a cell subcomplex $Y$ of $X$  
which is a union of normal bands of $X$ and non-frontier $1$--cells of $X$.
We require $Y$ to be maximal subject to 
the condition that every connected component of $Y$ is contractible.
Then $Y$ is a product cell complex
of a $1$--cell closure and a simplicial forest (possibly with isolated vertices).
Let $X'$ be the quotient of $X$ obtained by collapsing every component of $Y$ to a $1$--cell,
(via the projection onto $1$--cell factor with respect to the product structure).
Denote by $\mathrm{fr}'(X')$ the quotient of $\mathrm{fr}(X)$ in $X'$,
so $(X',\mathrm{fr}'(X'))$ is a cell complex pair with the quotient cell structure.
It follows by definition
$$\mathcal{Z}_2\left(X',\mathrm{fr}'(X')\right)\otimes\RR=\mathrm{ker}\left(
\mathcal{C}_2(X',\mathrm{fr}'(X'))\otimes\RR\stackrel{\partial'_2}{\longrightarrow}\mathcal{C}_1(X',\mathrm{fr}'(X'))\otimes\RR\right),$$
where $\partial'_2$ stands for the relative boundary operator of cellular chains.
Moreover, the collapsing map induces an isomorphism of $\RR$--linear spaces
$$\psi\colon\mathcal{Z}_2(X,\mathrm{fr}(X))\otimes\RR\to \mathcal{Z}_2(X',\mathrm{fr}'(X'))\otimes\RR.$$

For any $2$--cycle $z'\in \mathcal{Z}_2(X',\mathrm{fr}'(X'))$,
the $2$--cycle $\psi^{-1}(z')\in \mathcal{Z}_2(X,\mathrm{fr}(X))$
can be described as follows.
Let $c'\in\mathcal{C}_2(X,\mathrm{fr}(X))$ be the unique $2$--chain obtained by lifting every $2$--cell
of $X'$ to $X$. Then $\partial_2 c'$ lies in $\mathcal{C}_1(X,\mathrm{fr}(X))$,
and $-\partial_2 c'$ is a boundary of a unique $2$--chain $c''\in\mathcal{C}_2(X,\mathrm{fr}(X))$
formed by $2$--cells of $Y$. Hence the $2$--cycle $\psi^{-1}(z')$ equals $c'+c''$.

Denote by $m$ the number of monkey handles in $X$. We observe $0\leq m\leq l$.
We also observe that the number of components $b_0(Y)$ of $Y$
is at most $2m+1$, since $X$ is connected.
(Adding any extra half disk or normal band of $X$ to $Y$ will result in 
a cell complex of $b_0$ unchanged,
while adding any monkey handle may cause decrease of $b_0$ by at most $2$.)
Fix an ordering of the standard basis of $\mathcal{C}_i(X',\partial X')$.
Then the operator $\partial_2'$ is represented as a matrix of entries in $\ZZ$.
(We consider the matrix as acting by left multiplication on column vectors.)
By construction, 
the column vectors of $\partial_2'$ have $\ell^1$--norm at most $3$.
Here the $\ell^1$--norm refers to the sum of entry absolute values,
which agrees with the $\ell^1$--norm of $\mathcal{C}_1(X',\partial X')\otimes\RR$
with respect to the standard basis.
The number of rows for $\partial_2'$ is equal to $b_0(Y)$.
In general, for any $p\times q$ matrix $B$ of rank $r$ over $\RR$,
it is an elementary exercise of linear algebra to show that
the kernel of $B$ is spanned over $\RR$ by all the vectors in $\RR^q$ of the form
$$\sum_{t=0}^r (-1)^t\mathrm{det}\left(B\left[j_1,\cdots,j_r;i_1,\cdots,\widehat{i_t},\cdots,i_r\right]\right) \vec{e}_{i_t},$$
where the row indices $1\leq j_1\leq \cdots\leq j_r\leq p$ and the column indices $1\leq i_0<\cdots<i_r\leq q$
range over all possible choices.
Here $B[j_1,\cdots,j_r;i_1,\cdots,\widehat{i_t},\cdots,i_r]$ stands for the $r\times r$--block of $B$
indicated by the rows and columns,
and $\vec{e}_1,\cdots,\vec{e}_q\in\RR^q$ stand for the standard basis vectors.
Therefore, the kernel of $\partial_2'$ is spanned over $\RR$ 
by some $\ZZ$--elements 
$z'_1,\cdots,z'_s\in\mathcal{Z}_2(X',\mathrm{fr}'(X'))$
of $\ell^1$--norm bounded by 
$3^{\mathrm{rank}(\partial_2')}\times (\mathrm{rank}(\partial_2')+1)\leq 3^{2m+1}(2m+2)$,
using the fact $\mathrm{rank}(\partial_2')\leq b_0(Y)\leq 2m+1$.
Moreover, for every $z'\in\mathcal{Z}_2(X',\mathrm{fr}'(X'))$,
the above description of $\psi^{-1}(z')=c'+c''$
implies $\|\partial_2 c'\|_1\leq 3\times \|z'\|_1$ 
(because any handle of $X$ is attached to at most $3$ non-frontier $1$--cells)
and $\|c''\|_\infty\leq(1/2)\times\|\partial_2 c'\|_1$ 
(because $Y$ is a product of a $1$--cell closure with a simplicial forest).
It follows that $\|\psi^{-1}(z')\|_\infty= \max\{\|c'\|_\infty,\|c''\|_\infty\}
\leq (3/2)\times\|z'\|_1$.
Combining the above estimates,
we see that $\mathcal{Z}_2(X,\mathrm{fr}(X))\otimes\RR$
admits a spanning set of $\ZZ$--elements
$\psi^{-1}(z'_1),\cdots,\psi^{-1}(z'_s)$
of $\ell^\infty$--norm bounded by 
$(3/2)\times 3^{2m+1}(2m+2)=3^{2m+2}(m+1)\leq 3^{2l+2}(l+1)$.
This completes the proof of the combinatorial claim,
so Lemma \ref{generatingSet} follows.
\end{proof}

\subsubsection{Local effects of domination near vertices and edges}
We study the local homological behavior of domination 
over the geometrically defined regions $\mathcal{W}_e$ and $\mathcal{W}_v$. 

\begin{lemma}\label{dominationOnto}
Let $\beta>0$ be a constant and $\epsilon^*>0$ be a proper Margulis number for $\Hyp$.
With the assumptions and notations of Lemma \ref{pullStraight},
suppose that $N$ is endowed with some $\epsilon$--partially geometric Riemannian metric
and that $f\colon M\to N$ satisfies the listed conditions of Lemma \ref{pullStraight}.
If $f$ has nonzero degree, then the induced homomorphism
$$f|_*: H_2\left(M^{(2)}_{\mathcal{W}_v},M^{(2)}_{\partial\mathcal{W}_v};\RR\right)
\to H_2\left(\mathcal{W}_v,\partial\mathcal{W}_v;\RR\right)$$ 
is surjective for any vertex $v\in\vrtx(\Lambda)$.
The same statement holds true for any edge $e\in\edge(\Lambda)$ in place of $v$.
\end{lemma}

\begin{proof} 
This follows essentially from the Poincar\'{e}--Lefschetz duality.
We only prove the vertex case. The edge case is similar.

Given any vertex $v\in\vrtx(\Lambda)$,
observe that $f|_*$ equals the composition of the homomorphisms
$$\begin{CD}
	H_2\left(M^{(2)}_{\mathcal{W}_v},M^{(2)}_{\partial\mathcal{W}_v};\RR\right)@.@.\\
	@V\cong VV@.@.\\
	H_2\left(M^{(2)},M^{(2)}_{N\setminus\mathrm{int}({\mathcal{W}}_v});\RR\right)	@>i_* >>
	H_2\left(M,M_{N\setminus\mathring{\mathcal{W}}_v};\RR\right) @>\bar{f}_* >>
	H_2\left(N,N\setminus\mathring{\mathcal{W}}_v;\RR\right)\\
	@.@.@V\cong VV\\
	@.@. H_2\left(\mathcal{W}_v,\partial\mathcal{W}_v;\RR\right)
\end{CD}$$
where the vertical isomorphisms are homology excisions,
and where $i_*$ is induced by inclusion of the $2$--skeleton.

We show that $i_*$ can be made surjective 
by modifying $f$ only on the interior of the $3$--simplices of $M$.
Note that such modification preserves the homotopy class of $f$,
since $N$ is an Eilenberg--MacLane space.

To this end, consider any $3$--simplex $B$ of $M$ in the given simplicial decomposition.
The sphere $\partial B$ is a union of (possibly empty) embedded regions 
$\partial B_{\mathcal{W}_v}$ and $\partial B_{N\setminus\mathring{\mathcal{W}}_v}$,
meeting along a union of mutually disjoint 
simple closed curves $\partial B_{\partial\mathcal{W}_v}$,
(because it is the intersection of $\partial B$ with the frontier $M^{(2)}_{\partial\mathcal{W}_v}$ of 
the complete handle complex $M^{(2)}_{\mathcal{W}_v}$).
The curves $\partial B_{\partial\mathcal{W}_v}$ 
bound some mutually disjoint properly embedded disks in $B$.
We choose a union of such disks,
and modify the map $f$ in the interior of $B$. 
We require the disks to be mapped to $\partial\mathcal{W}_v$,
and the complementary components in $B$ of the disk union
to be mapped to either $\mathcal{W}_v$ or $N\setminus\mathring{\mathcal{W}}_v$.
This can be done because $\partial\mathcal{W}_v$,
$\mathcal{W}_v$, and $N\setminus\mathring{\mathcal{W}}_v$ are all Eilenberg--MacLane spaces
and because $\partial\mathcal{W}_v$
is incompressible in $N$.

By performing the modification of $f$ on the interior of every $3$--simplex,
we may require henceforth that
$M$ is obtained from
$M^{(2)}\cup M_{N\setminus\mathring{\mathcal{W}}_v}$
by attaching $3$--cells. 
(According to the above construction,
the $3$--cells are precisely
 those disk complements which are mapped to $\mathcal{W}_v$.)
Then the inclusion of pairs induces an epimorphism
$H_2(M^{(2)}\cup M_{N\setminus\mathring{\mathcal{W}}_v},
M_{N\setminus\mathring{\mathcal{W}}_v};\RR)
\to H_2(M,M_{N\setminus\mathring{\mathcal{W}}_v};\RR)$.
Observe also the excision isomorphism
$H_2(M^{(2)},M^{(2)}_{N\setminus\mathring{\mathcal{W}}_v};\RR)
\cong H_2(M^{(2)}\cup M_{N\setminus\mathring{\mathcal{W}}_v},
M_{N\setminus\mathring{\mathcal{W}}_v};\RR)$.
By composing these homomorphisms, we see that $i_*$ is surjective
for the modified $f$.

It remains to show
the surjectivity of $\bar{f}_*$ in the above decomposition of $f|_*$.
To this end, observe that the mapping degree of $f$ is unchanged under the modification,
since it is homotopy invariant.
The nonzero degree assumption for $f$ implies
the injectivity of the induced homomorphism $\bar{f}^*$ 
on the third $\RR$--coefficient relative cohomology,
by the commutative diagram of homomorphisms
$$\begin{CD}
	H^3\left(N,N\setminus\mathring{\mathcal{W}}_v;\RR\right)	@>\bar{f}^*>>	H^3\left(M,M_{N\setminus\mathring{\mathcal{W}}_v};\RR\right)\\
	@VVV@VVV\\
	H^3(N;\RR)	@>f^*>>	H^3(M;\RR).
\end{CD}$$
We see that the induced homomorphism
$$\bar{f}^*:  
H^*\left(N,N\setminus\mathring{\mathcal{W}}_v;\RR\right)\to 
H^*\left(M,M_{N\setminus\mathring{\mathcal{W}}_v};\RR\right)$$
is injective for every dimension, 
because of the commutative diagram of homomorphisms
$$\begin{CD}
	H^i\left(N-N\setminus\mathring{\mathcal{W}}_v;\RR\right)	@.\times @.H^{3-i}\left(N,N\setminus\mathring{\mathcal{W}}_v;\RR\right)	@>\smile>>
								H^3\left(N,N\setminus\mathring{\mathcal{W}}_v;\RR\right)\\
 	@V \bar{f}^* VV @. @V \bar{f}^* VV @V \bar{f}^* VV\\
	H^i\left(M-M_{N\setminus\mathring{\mathcal{W}}_v};\RR\right)	@.\times	@. H^{3-i}\left(M,M_{N\setminus\mathring{\mathcal{W}}_v};\RR\right)
					@>\smile>> H^3\left(M,M_{N\setminus\mathring{\mathcal{W}}_v};\RR\right),\\
\end{CD}$$
 where the cup-product pairings are nonsingular and 
where the rightmost
vertical homomorphism is injective.
It follows that the homomorphism $\bar{f}_*$ 
is surjective for every dimension, and in particular, for dimension $2$.

The surjectivity of $i_*$ and $\bar{f}_*$ implies the surjectivity
of the homomorphism $f|_*$, as asserted.
\end{proof}

\subsubsection{Summary of proofs}
We complete the proofs of Propositions \ref{distortionDomination}, \ref{looplessGraphCase},
and Theorem \ref{main-dominate} as follows.

\begin{proof}[{Proof of Proposition \ref{looplessGraphCase}}] 
Suppose that $N$ is an orientable closed irreducible nongeometric $3$--manifold 
obtained from a nondegenerate gluing $\phi$ of a preglue graph of geometrics $(\Lambda,\pieces)$,
whose graph is loopless and entire, 
as assumed in Proposition \ref{looplessGraphCase}.
We adopt the notations of Section \ref{Subsubsec-partialGeometrization}.
Suppose that $M$ is an orientable closed $3$--manifold
and that $f\colon M\to N$ is a nonzero degree map.
Take a simplicial decomposition of $M$ with the fewest possible $3$--simplices,
so the number of $3$--simplices realizes the triangulation number $t=t(M)$.
We show that there exists some constant $C=C(t)>0$,
such that $\distortion_v(\phi)<C$ holds for any vertex $v\in\vrtx(\Lambda)$,
and that $\distortion_e(\phi)<C$ holds for any edge $e\in\edge(\Lambda)$.
We only argue for the vertex case. The edge case is completely similar.

Let $v\in\vrtx(\Lambda)$ be a vertex.
Fix some constant $\beta>0$ and some proper Margulis number $\epsilon^*>0$ for $\Hyp$.
(For example, take $\beta=1$ and $\epsilon^*$ smaller than the Margulis constant $\epsilon_3$ for $\Hyp$.)
With the notations and the assumptions of Lemma \ref{pullStraight},
we furnish $N$ with some $\epsilon$--partially geometric Riemannian metric,
and pull straight the nonzero degree map $f\colon M\to N$,
so that it satisfies the listed conditions of Lemma \ref{pullStraight}.
There is a spanning subset $[S_1],\cdots,[S_m]$
of the relative homology  $H_2(M^{(2)}_{\mathcal{W}_v}, M^{(2)}_{\partial\mathcal{W}_v};\RR)$,
represented by relative $\ZZ$--cycles of area individually bounded by 
$A(\beta,t)$, as guaranteed by Lemma \ref{generatingSet}.
In fact, the relative $\ZZ$--cycles can be regarded as 
proper immersions of compact oriented surfaces
$j_i\colon (S_i,\partial S_i)\to (\mathcal{W}_v,\partial\mathcal{W}_v)$,
for all $1\leq i\leq m$.
(This follows from the construction by gluing up handles.)
By Lemma \ref{isoperimetric}, we estimate, 
for all $1\leq i\leq m$:
\begin{eqnarray*}
\sqrt{\qf_\phi({j_{i*}}[\partial S_i])}&\leq&\frac14\,\left(\sinh\left(\frac{\epsilon^*}2\right)-\sinh\left(\frac{\epsilon}2\right)\right)^{-1}\cdot\Area(j_i(S_i))\\
&\leq& \frac14\cdot\left(\sinh\left(\frac{\epsilon^*}2\right)-\sinh\left(\frac{\epsilon}2\right)\right)^{-1}\cdot A(\beta,t).
\end{eqnarray*}
Note $j_{i*}[\partial S_i]=\partial_*j_{i*}[S_i]$
in $\partial_*H_2(\mathcal{W}_v,\partial\mathcal{W}_v;\ZZ)$.
Since $f$ has nonzero degree,
the elements $j_{1*}[\partial S_1],\cdots,j_{m*}[\partial S_m]$
form a spanning subset of 
$\partial_*H_2(\mathcal{W}_v,\partial\mathcal{W}_v;\RR)$,
by Lemma \ref{dominationOnto}.
Then we estimate, by Lemma \ref{largeGenerator}:
$$\distortion_v(\phi)
\leq \max_{1\leq i\leq m}\sqrt{\qf_\phi({j_{i*}}[\partial S_i])}
\leq \frac14\cdot\left(\sinh\left(\frac{\epsilon^*}2\right)-\sinh\left(\frac{\epsilon}2\right)\right)^{-1}\cdot A(\beta,t).$$
As $0<\epsilon<\epsilon^*$ can be arbitrarily small, 
we obtain an estimate as desired:
$$\distortion_v(\phi)\leq \frac{A(\beta,t)}{4\sinh\left(\frac{\epsilon^*}2\right)},$$
where the right-hand side depends only on $t$.
(In fact, one can obtain an improved estimate $\distortion_v(\phi)\leq A(\beta,t)$ 
by Remark \ref{isoperimetricRemark}.
Moreover,
if one uses the expression of $A(\beta,t)$ in Remark \ref{A_expression},
one may replace $\beta$ with $0$ since it can be taken arbitrarily small.)
\end{proof}

\begin{proof}[{Proof of Proposition \ref{distortionDomination}}]
Combine Lemma \ref{reductionLooplessGraph} and Proposition \ref{looplessGraphCase}.
\end{proof}


\begin{proof}[{Proof of Theorem \ref{main-dominate}}] 
To summarize, let $M$ be an orientable closed $3$--manifold.
By Corollary \ref{fixing_preglue_corollary},
we may assume a fixed nontrivial preglue graph of geometrics $(\Lambda,\pieces)$,
and 
it suffices to bound the equivalence classes of nondegenerate gluings $\phi\in\Phi(\Lambda,\pieces)$
for which $M$ dominates the glued-up nongeometric $3$--manifold $N_\phi$.
By Proposition \ref{distortionDomination}, 
the primary distortion $\distortion(N_\phi)=\distortion_\Lambda(\phi)$ is bounded 
by some constant determined by (the homeomorphism type of) $M$.
By Proposition \ref{finiteGluings}, 
there are at most finitely many equivalence classes of nondegenerate gluings
whose primary distortion is bounded by a given constant. 
Therefore, there are at most finitely many equivalence classes of nondegenerate gluings
in $\Phi(\Lambda,\pieces)$ and $\phi$ belongs to one of them.
This completes the proof of Theorem \ref{main-dominate}.\end{proof}

\section{Domination of bounded degree}\label{Sec-boundedDegree} 
In this section, we prove Corollary \ref{main-d-dominate}. 

\begin{lemma}\label{reductionSF} 
The statement of Corollary \ref{main-d-dominate} holds true if it holds true
under the following assumption:
The target is a closed orientable Seifert fibered $3$--manifold 
over an orientable base $2$--orbifold 
and with nonvanishing Euler classes.\end{lemma}

\begin{proof} 
We can reduce the statement of Corollary \ref{main-d-dominate} to the case
when the target is irreducible, 
because any orientable closed $3$--manifold $1$--dominates any 
of its connected-sum factors in the Kneser--Milnor decomposition,
and because the number of connect-sum factors 
in the target is bounded in terms of the Kneser--Haken number of the source
(Lemma \ref{LemmaBRW}).
By Theorems \ref{main-dominate}, \ref{ThmBRW} and \ref{ThmBBW},
it remains to prove the statement assuming that the target
supports one of the geometries $\Sph$, $\Nil$ or $\Sft$. 
Such targets are precisely 
closed orientable Seifert fibered $3$--manifolds 
with nonvanishing Euler classes. 
To further reduce to the case when the base $2$--orbifold
is orientable, 
we pass to a $2$--fold covering space, getting rid of essentially
embedded Klein bottles,
following a similar argument as used in Lemma \ref{reductionLooplessGraph}.\end{proof}


\begin{proposition}\label{d-dominateSF} 
Given any nonzero integer $d$ up to sign, 
every orientable closed $3$--manifold $d$--dominates
at most finitely many Seifert fibered $3$--manifolds 
with orientable base $2$--orbifolds and with nontrivial Euler classes.\end{proposition}

Proposition \ref{d-dominateSF} is known when $d$ equals $1$, 
due to Hayat-Legrand--Wang--Zieschang \cite{HWZ} for the $\Sph$--geometric case,
and Wang--Zhou \cite{WZ} for the $\Nil$--geometric case and the $\Sft$--geometric case.
We generalize their arguments by 
establishing the following promoted lemma:

\begin{lemma}[{Compare \cite[Lemma 3]{HWZ} and \cite[Lemma 3 (1)]{WZ}}]\label{boundTor} 
For any nonzero integer $d$ up to sign, 
if an orientable closed $3$--manifold $M$ $d$--dominates an orientable closed $3$--manifold $N$, 
then the following comparison holds true:
$$|\Tor\,H_1(N;\ZZ)|\leq d\cdot|H_1(M;\ZZ/d\ZZ)|\cdot|\Tor\,H_1(M;\ZZ)|,$$
where $\Tor$ stands for the torsion submodule, 
and $|\cdot|$ stands for cardinality.\end{lemma}  

\begin{proof} 
Suppose that $f\colon M\to N$ is a map of degree $d\neq0$
(after fixing orientations for $M$ and $N$).
The umkehr homomorphism:
$$f_!\colon H_*(N;\ZZ)\to H_*(M;\ZZ),$$
is defined by
$f_!(\alpha)= [M]\frown f^*(\check\alpha)$ for all $\alpha\in H_*(N;\ZZ)$, 
where $\check\alpha\in H^{3-*}(N;\ZZ)$
stands for the Poincar\'{e} dual of $\alpha$. 
It is straightforward to check that 
the composition $f_*\circ f_!\colon H_*(N;\ZZ)\to H_*(N;\ZZ)$ is 
the scalar multiplication by $d$.
In particular,
the submodule $d\cdot\Tor\,H_1(N;\ZZ)$ is surjected by $f_!(\Tor\,H_1(N;\ZZ))\leq \Tor\,H_1(M;\ZZ)$.
On the other hand, from the long exact sequence:
$$\cdots\longrightarrow H_1(N;\ZZ)\stackrel{d}\longrightarrow H_1(N;\ZZ)\longrightarrow H_1(N;\ZZ/d\ZZ)\longrightarrow 0,$$
we see $\Tor\,H_1(N;\ZZ)/(d\cdot\Tor\,H_1(N;\ZZ))\leq H_1(N;\ZZ/d\ZZ)$.
Since $f\colon M\to N$ has degree $d$,
the image of $H_1(M;\ZZ/d\ZZ)$ in $H_1(N;\ZZ/d\ZZ)$ has index at most $d$. 
This yields the asserted inequality.
\end{proof}

\begin{proof}[Proof of Proposition \ref{d-dominateSF}] 
The proof is the same as in \cite{HWZ,WZ},
only with an update with Lemma \ref{boundTor}.
For the reader's reference, 
we provide a brief outline as follows. 

We adopt the notation $N=\Sigma(g;b_0,b_1/a_1,\cdots,b_s/a_s)$
for an orientable closed Seifert fibered $3$--manifold, 
where $s,g\geq0$ and $b_0$ are integers, 
and where $0<b_i<a_i$ are coprime integers for $1\leq i\leq s$.
It means that $N$ has an orientable base $2$--orbifold $F_g(a_1,\cdots,a_s)$, 
namely, 
an orientable closed surface of genus $g$ and with cone-points of orders $a_i$.
The orbifold Euler characteristic of the base $2$--orbifold is
$$\chi=2-2g-\sum_{i=1}^s(1-\frac1{a_i}).$$
The Euler class of the Seifert fibration (as a rational number) is
$$e=-b_0-\sum_{i=1}^s\,\frac{b_i}{a_i}.$$ 
The torsion size of the first homology is
$$|\Tor\,H_1(N;\ZZ)|=|e|\cdot\prod_{i=1}^s a_i,$$
assuming $e\neq0$.

Suppose that $M$ is an orientable closed $3$--manifold which 
$d$--dominates some $N=\Sigma(g;b_0,b_1/a_1,\cdots,b_s/a_s)$ 
as above with $e\neq 0$. 
Following \cite{HWZ} and \cite{WZ}, we consider cases 
according to the sign of $\chi$.

When $\chi>0$ occurs, 
we have $g=0$ and $s\leq 3$. 
For $0\leq s\leq 2$, $N$ is a lens space (possibly a $3$--sphere), 
so there are only finitely many allowable homeomorphism types for $N$ by \cite[Corollary 1]{HWZ}.
(The argument uses the linking paring on $\Tor\,H_1(N;\ZZ)$.)
For $s=3$, $N$ must be a prism $3$--manifold
$\Sigma(0;b_0,1/2,1/2,b_3/a_3)$, 
or $\Sigma(g;b_0,b_2/3,b_3/3)$,
or $\Sigma(g;b_0,1/2,b_2/3,b_3/4)$, 
or $\Sigma(g;b_0,1/2,2/3,b_3/5)$. For the latter
three types, $b_2,b_3$ are automatically bounded by their denominators, 
so $b_0$ can be bounded with Lemma \ref{boundTor} 
and the above formulas. 
For the prism case, $N$ admits 
a $\ZZ/2\ZZ$--action whose quotient is a lens space $L$.
The finite cyclic group $\pi_1(L)$ has order
$|2a_3b_0+2b_3-2a_3|$ if $a_3$ is odd,
or $|a_3b_0+b_3-a_3|$ if $a_3$ is even.
Then $|a_3b_0+b_3-a_3|$ are bounded by the lens space case above,
since $M$ $2d$--dominates $L$. 
As the torsion-size comparison bounds $|a_3b_0+b_3+a_3|$, 
noticing $0<b_3<a_3$, 
we have upper bounds for $|b_0|,a_3,b_3$ when $N$ is prism.
Thus there are at most finitely many homeomorphically distinct $N$ with $\chi>0$. 

When $\chi=0$ occurs, 
there are only finitely many allowable values of $s,g$ and $a_1,\cdots,a_s$, 
by the formula of $\chi$. 
For each possibility, 
there are only finitely allowable values for $b_1,\cdots,b_s$, by $0<b_i<a_i$.
The value $b_0$ can be bounded with the torsion-size comparison. 
Thus there are at most finitely many homeomorphically distinct $N$ with $\chi=0$.

When $\chi<0$ occurs, 
we have $\chi\leq-1/42$,
equal if the base $2$--orbifold is a turnover $F_0(2,3,7)$. 
Using Seifert volume as introduced by Brooks--Goldman \cite{BG},
(denoting by $\mathrm{SV}$,)
we have $\mathrm{SV}(M)\geq d\cdot\mathrm{SV}(N)=|e|^{-1}\chi^2d$, (see \cite[Lemmas 3(2) and 4(3)]{WZ}). 
Then $|e|$ is bounded from zero in terms of $M$. 
By Lemma \ref{boundTor} and the above torsion-size formula, 
we can bound $s$ and $a_1,\cdots,a_s$. 
This in turn yields an upper bound for $|b_0|$,
again by the torsion-size comparison. 
Thus there are at most finitely
many homeomorphically distinct $N$ with $\chi<0$ as well. 

This completes the proof of Proposition \ref{d-dominateSF}.
\end{proof}

\begin{proof}[Proof of Corollary \ref{main-d-dominate}] Combine Lemma \ref{reductionSF} and Proposition \ref{d-dominateSF}.\end{proof}

\section{Conclusion}\label{Sec-Conclusion}
Both Theorem \ref{main-dominate} and Corollary \ref{main-d-dominate} 
are results about maps between $3$--manifolds. 
We propose some questions for further study.

\begin{question}\label{questionEpimorphism} Which groups surject 
at most finitely many isomorphically distinct fundamental groups of aspherical $3$--manifolds?\end{question}

Such groups are sometimes called \emph{tiny groups}.
For example, virtually solvable groups are tiny groups. 
Our techniques seem to imply that
finitely generated groups with vanishing first Betti number are also tiny.

Let $M$ be an orientable closed $3$--manifold. For any nonzero integer $d$ up to sign,
denote by $\tau_M(d)$ the maximal number of homeomorphically distinct $3$--manifolds 
which are dominated by $M$ with degree at most $d$. 
By Corollary \ref{main-d-dominate}, $\tau_M(d)$ is a positive finite integer.

\begin{question}\label{questionGrowth} 
Is it possible to decide $\tau_M$ for $|d|$ sufficiently large?\end{question}

An even harder problem is
to decide whether a given orientable closed $3$--manifold $M$
dominates (or $d$--dominates) another given $N$ \cite[Question 1.1]{Wa}. 

\begin{question}\label{questionDistortion} Given an orientable closed nontrivial
graph manifold $N$, is there an explicit bound for the Seifert volume $\mathrm{SV}(N)$
in terms of its graph $\Lambda(N)$ and its average distortions $\distortion_v(N)$
and $\distortion_e(N)$?\end{question}

The question is motivated by the phenomenon that 
Seifert volume also reflects the complexity of gluings.
More generally, we wonder if there is a more insightful notion
of global distortion, other than the primary average distorion
we have introduced.

\bibliographystyle{amsalpha}

\end{document}